\documentclass[12pt]{amsart}
\usepackage{xcolor}
\usepackage{amsmath}
\usepackage{amsfonts}
\usepackage{geometry}
\usepackage{xy}
\input xy
\xyoption{all}
\newcommand{\supp}{{\operatorname{Supp}}}

\newcommand{\cliff}{{\textrm {Cliff}}}
\newcommand{\Cliff}{{\textrm {Cliff}}}
\newcommand{\cA}{{\mathcal{A} }}
\newcommand{\cB}{{\mathcal{B}}}

\newcommand{\cK}{{\mathcal{K} }}
\newcommand{\cM}{{\mathcal{M} }}

\def\cala{\mathcal A}
\def\calb{\mathfrak B}

\def\calk{\mathcal K}
\def\call{\mathcal L}

\def\calq{\mathcal Q}

\def\bbc{\mathbb C}

\def\bbn{\mathbb N}

\def\bbr{\mathbb R}

\newcommand{\pdg}{P_d(\Gamma)}
\newcommand{\matr}[4]{
\left(
\begin{array}{cc}
#1 & #2  \\
#3 & #4
\end{array}
\right)
}
\newtheorem{thm}{THEOREM}
\newtheorem{prop}[thm]{PROPOSITION}
\newtheorem{conj}[thm]{CONJECTURE}
\newtheorem{lem}[thm]{LEMMA}

\theoremstyle{definition}
\newtheorem{defn}[thm]{{DEFINITION}}

\newcommand{\set}[1]{\{#1\}}
\newcommand{\norm}[1]{\|#1\|}
\def\rank{\operatorname{rank}}
\def\calgp{B_{alg}^p}
\def\cp{B^p}
\def\clalgp{B_{L,alg}^p}
\def\clp{B_{L}^p}
\def\lp{\ell^p}

\def\calbp{\calb_{\infty,p}}
\def\cBp{\cB}

\newcommand{\cspace}{\cp(\pdg,\cA)}
\newcommand{\clspace}{\clp(\pdg,\cA)}

\title[The coarse geometric $\lp$-Novikov Conjecture]
        {The coarse geometric $\lp$-Novikov conjecture for subspaces of non-positively curved manifolds}
\thanks{This work is supported in part by the National Natural Science Foundation of China Nos. 11771143, 11831006.}
\author{Lin Shan}
\address{Lin Shan: College of Natural Sciences\\
Department of Mathematics\\
17 University Ave. Ste 1701\\
  University of Puerto Rico\\
  San Juan PR, 00925-2537}
\email{lin.shan@upr.edu}
\author{Qin Wang}
\address{Qin Wang: Research Center for Operator Algebras \\
School of Mathematical Sciences \\
East China Normal University \\
Shanghai  200241 \\
P. R. China.}
\email{qwang@math.ecnu.edu.cn}

\date{}

\begin{document}

\begin{abstract}
In this paper, we prove the coarse geometric $\lp$-Novikov conjecture for  metric spaces with bounded geometry
which admit a coarse embedding into a simply connected complete Riemannian manifold of nonpositive sectional
curvature.
\end{abstract}

\maketitle

\baselineskip=1.5\baselineskip

\section{Introduction}

The coarse geometric Novikov conjecture \cite{HR,Yu95,Yu98,KY06} is a statement that the coarse Baum-Connes assembly map  from the coarse $K$-homology of a metric space to the $K$-theory of the Roe $C^*$-algebra, which encodes the coarse geometry of the space, is injective.
This is a geometric analogue of the strong Novikov conjecture and provides an algorithm to determine the non-vanishing problem of the
higher index of the Dirac operator on a noncompact complete
Riemannian manifold. It implies Gromov's conjecture that a uniformly contractible Riemannian manifold with bounded geometry cannot have uniformly positive scalar curvature, and the zero-in-the-spectrum conjecture stating that the Laplacian operator acting on the space of all $L^2$-forms of a uniformly contractible Riemannian manifold has zero in its spectrum.
\par
A remarkable progress was achieved by G. Yu who proved the
coarse Baum-Connes conjecture, and consequently the coarse geometric Novikov conjecture, for metric spaces with bounded
geometry which admit a coarse
embedding into a Hilbert space \cite{Yu00}. Among the main tools in \cite{Yu00} are the localization algebra of Yu \cite{Yu97}
together with the twisted Roe algebra technique. A fundamental idea underlining the approach in \cite{Yu00} is that the index of a Dirac
operator is more computable if the
Dirac operator is twisted by a family of ``almost flat Bott bundles". This approach inspires several later progresses on the coarse geometric Novikov conjecture for coarse embeddings into certain Banach spaces \cite{KY06, CWY} or non-positively curved manifolds \cite{SW}.

Recently, an $\ell^p$-analog of the coarse geometric Baum-Connes assembly map for $1<p<\infty$ was introduced in
\cite{ChungNowak}; see also \cite{ZZ}.
An important impetus behind this generalization is the unpublished work of G. Kasparov and G. Yu on
the $L^p$-Novikov and Baum-Connes conjectures (cf. \cite{Kas2013}), together with earlier works of Lafforgue's Banach $KK$-theory \cite{Lafforgue} and the discovery of G. Yu \cite{Yu05} that all Gromov hyperbolic groups, which include plenty of groups with Kazhdan's property (T), admit a proper affine isometric action on an $\ell^p$-space for some $p\geq 2$. Other similar $L^p$-assembly map has been considered in \cite{Chung} by Y. C. Chung. And closely related to these problems, rigidity and
$K$-theory of $\ell^p$-Roe-type algebras have also be studied by Y. C. Chung and K. Li \cite{ChungLi18,ChungLi19}.

The $\ell^p$-version of the geometric Novikov conjecture is a natural analog of the classical conjecture obtained by considering algebras of operators on $\ell^p$-spaces. While applications to geometry and topology have yet to be found when $p\neq 2$, there is motivation in the
coarse geometric $\ell^p$-Novikov conjecture coming from comparison with the classical case, and the intrinsic interest in comparing $K$-theories of
different completions of a given algebra.

In this paper, we shall prove the following result.
\begin{thm}\label{main}
Let $\Gamma$ be a discrete metric space with bounded geometry. If $\Gamma$ admits a coarse embedding into a simply-connected complete Riemannian manifold of non-positive sectional curvature, then the coarse geometric $\ell^p$-Novikov conjecture holds for $\Gamma$, i.e., the assembly map
$$
\mu:\lim_{d\to\infty}K_*(\clp(\pdg))\to K_*(\cp(\Gamma))
$$
is injective for all $1<p<\infty$.
\end{thm}
Recall that a map $f:X\to Y$ for from a metric space
$X$ to another $Y$  is said to be a {\em coarse embedding} \cite{Grom} if there exist non-decreasing functions $\rho_1$ and $\rho_2$
from $\mathbb R_+=[0,\infty)$ to $\mathbb R$ with
$\lim_{r\to\infty}\rho_i(r)=\infty$ for $i=1,2$, such that
$$\rho_1(d(x,y))\leq d(f(x),f(y))\leq \rho_2(d(x,y))$$
 for all $x, y\in X$.
The above assembly map $\mu$ is induced by the evaluation-at-zero map $e$ from the localization $\ell^p$ algebra $\clp(\pdg)$ of the
Rips complex of $\Gamma$ to the $\ell^p$-Roe algebra $\cp(\Gamma)$ of $\Gamma$. The definition of the $\ell^p$-assembly map is motivated by
the result of G. Yu in \cite{Yu97} that the local index map from $K$-homology to the $K$-theory of the localization algebra is an
isomorphism for a finite dimensional
simplicial complex. Due to the local nature, it can be shown (cf. \cite{ZZ}) that the $K$-theory of the $\ell^p$-localization algebras
$\clp(\pdg)$ is independent of the choice of $1<p<\infty$. Therefore, the left-hand side of the assembly map $\mu$ in Theorem 1 is
isomorphic to the classical coarse $K$-homology of the space $\Gamma$.

The proof of Theorem 1 is again based on the fundamental idea and tools in \cite{Yu00} of G. Yu by using localization algebra technique and an $\ell^p$-version of the twisted Roe algebras and $\ell^p$-Bott maps. We closely follow our previous work \cite{SW} in the classical $p=2$ case,
with necessary technical adjustments.

It should be noted that techniques used in the $C^*$-algebraic setting often do not transfer to the $L^p$-setting in a straightforward manner. This is due to the more complicated geometry of $L^p$-spaces, including the fact that they are not reflexive unless $p=2$. For instance, while for any closed two-sided ideals $I, J$ in a $C^*$-algebra $A$ we always have $I\cap J=IJ$
(this general fact is frequently used to establish the $K$-theory Mayer-Vietoris exact sequences),
this equality may not hold in an arbitrary $L^p$-operator algebra (as a clue, consider $\mathbb{C}$ with its usual norm and the trivial product given
by $xy=0$ for all $x, y\in \mathbb{C}$). In general, an $L^p$-operator algebra need not have a (contractive, one-sided) approximate identity. However, we will show that closed ideals in the $\ell^p$-Roe algebras or the twisted $\ell^p$-Roe algebras
supported on subspaces of the metric space $\Gamma$ or open subsets of the manifold $M$ do admit contractive approximate units. This allows us to
establish $K$-theory Mayer-Vietoris sequences for the $\ell^p$-Roe algebra and the twisted $\ell^p$-Roe algebra.

Another subtle issue is about tensor products associated with $L^p$-spaces. In general,
the tensor product $T\otimes S$ of a bounded operator $T$ on $L^p(\mu)$  and a bounded operator $S$ on a Banach space $E$ may not extend
to a bounded operator on the ``natural tensor product" $L^p(\mu)\otimes_{p} E$, unless for example $E=L^p(\nu)$ is another $L^p$-space,
in which case $\|T\otimes S\|=\|T\|\|S\|$.
This suggests us to view the algebra $\cA=C_0(M,\cliff_{\bbc}(TM))$ in the construction of the twisted $\ell^p$-Roe algebra and Bott elements
in $K$-theory as an $L^p$-operator algebra. Since the Clifford bundle $\cliff_{\bbc}(TM))$ is finite dimensional, one would naturally like to
regard it as an $\ell^p$-space bundle so that the algebra $\cA=C_0(M,\cliff_{\bbc}(TM))$ could act on the $L^p$-space $L^p(M,\cliff_{\bbc}(TM))$
of the $\ell^p$-space bundle $\cliff_{\bbc}(TM))$.
However, since the $\ell^p$-norm on a tangent space $T_x M$ depends on the choice of the (orthonormal) basis
of $T_xM$, if $M$ is not flat we can not end up with a consistent $\ell^p$-structure on the tangent bundle $TM$ or the Clifford bundle
$\cliff_{\bbc}(TM))$, which is needed for the
construction of ``the family of unformly almost flat Bott elements" on $M$. To solve this confliction, we will view
$\cA=C_0(M,\cliff_{\bbc}(TM))$ acting on $L^p(M,\cliff_{\bbc}(TM))$ which is the $L^p$-space of locally measurable sections of
{\em the Hilbert space bundle $\cliff_{\bbc}(TM))$}. It turns out that
the tensor product $T\otimes S$ of a bounded operator $T$ on $\ell^p$  and a bounded operator $S$ on
$L^p(M,\cliff_{\bbc}(TM))$, regarded as the $L^p$-space of the Hilbert space bundle $\cliff_{\bbc}(TM))$, still extends
to a bounded operator on the ``natural tensor product" $\ell^p \otimes_p L^p(M, \cliff_{\bbc}(TM))$
and satisfies $\|T\otimes S\|=\|T\|\|S\|$.

The paper is organized as follows. In section 2 we recall the $\ell^p$-Roe algebra, $\ell^p$-localization algebras and the coarse geometric
$\ell^p$-Novikov conjecture. In section 3, we study approximate units for an ideal of the $\ell^p$-Roe algebra supported on a subspace and
present an $\ell^p$-coarse Mayer-Vietoris principle. In section 4, we first discuss certain measure theory aspect of
the $L^p$-space $L^p(M, \cliff_{\bbc}(TM))$ of the Hilbert space bundle $\cliff_{\bbc}(TM)$ and the natural tensor norm $\otimes_p$, so as to
view $\cA=C_0(M,\cliff_{\bbc}(TM))$ as an $L^p$-operator algebra. Then we define the twisted $\ell^p$-Roe algebra and its
localization counterpart and discuss how to use ideals supported on separate open subsets of $M$ to
show that the evaluation map induces an isomorphism for twisted algebras. In section 5, we adapt Yu's arguments about strong Lipschitz homotopy
invariance to the $\ell^p$-setting. In section 6, we construct families of uniformly almost flat Bott generators to establish a Bott
map $\beta$ from the $K$-theory of the $\ell^p$-Roe
algebra to the $K$-theory of the twisted $\ell^p$-Roe algebra, and a Bott map $\beta_L$ between the corresponding $\ell^p$-localization algebras.
In section 7, we complete the proof of the main theorem of this paper.

\vskip 10 mm

\section{The coarse geometric Novikov conjecture}
In this section, we shall recall the concepts of the $\lp$-Roe algebra \cite{Roe93,ChungNowak}, Yu's $\lp$-localization algebras \cite{Yu97, ChungNowak} and the coarse geometric $\ell^p$-Novikov conjecture \cite{ChungNowak}.

Let $X$ be a proper metric space. Recall that the space $X$ is called {\em proper} if every closed ball is compact. When $X$ is discrete, we say that $X$ has {\em bounded geometry} if for any $R>0$, there exists $N_R>0$ such that for any $x\in X$ the cardinality $|B(x; R)|$ is less than or equal to $N_R$. For $r>0$, an $r$-net in $X$ is a discrete subset $Y\subset X$ such that for any $y_1,y_2\in Y$, $d(y_1,y_2)\geq r $ and for any $x\in X$ there is a $y\in Y$ such that $d(x,y)\leq r$. A general metric space $X$ is called to have bounded geometry if $X$ has an $r$-net $Y$ for some $r>0$ such that $Y$ has bounded geometry.

Throughout this paper, we assume $p>1$, and denote by $\calk_p=\calk(\ell^p)$, the Banach algebra of all compact operators on $\ell^p$.

\begin{defn}[\cite{Roe93,ChungNowak}]
Let $X$ be a proper metric space, and fix a countable dense subset $Z\subseteq X$. Let $T$ be a bounded operator on $\lp(Z,\lp)$, and write $T=(T(x,y))_{x,y\in Z}$ so that each $T(x,y)$ is a bounded operator on $\lp$. The operator $T$ is said to be {\em locally compact} if
\begin{itemize}
\item each $T(x,y)$ is a compact operator on $\lp$;
\item for every bounded subset $B\subseteq X$, the set
$$
\set{(x,y)\in(B\times B)\cap(Z\times Z):T(x,y)\neq 0}
$$
is finite.
\end{itemize}
The {\it propagation} of $T$ is defined to be
$$
propagation(T)=\inf\set{S>0:T(x,y)=0\text{ for all }x,y\in Z\text{ with }d(x,y)>S}.
$$
The {\it algebraic $\lp$ Roe algebra} of $X$, denoted by $\calgp(X)$, is the subalgebra of $\call(\lp(Z,\lp))$ consisting of all finite propagation, locally compact operators. The {\it $\lp$ Roe algebra} of $X$, denoted by $\cp(X)$, is the closure of $\calgp(X)$ in $\call(\lp(Z,\lp))$.
\end{defn}

Up to non-canonical isomorphisms, $\cp(X)$ does not depend on the choice of the dense subspace $Z$, while up to canonical isomorphism, its $K$-theory does not depend on the choice of $Z$. The proof in \cite{HRY} for $p=2$ works well for general $p>1$.

\begin{defn}[\cite{Yu97}]  The {\em $\lp$-localization algebra} $\clp(X)$ is the norm-closure of the algebra of all
bounded and uniformly norm-continuous functions $g: \; [0,\infty)\to \cp(X)$ such that
$$
propagation(g(t))\to 0\quad\mbox{as}\quad t\to \infty.
$$
\end{defn}

The evaluation homomorphism $e$ from $\clp(X)$ to $\cp(X)$ is defined by
$$e(g)=g(0)$$
for all $g\in \clp(X)$.

\begin{defn}[\cite{WY}]
Let $\Gamma$ be a locally finite metric space. Let $d\geq 0$. The {\em Rips complex of $\Gamma$ at scale $d$}, denoted by $P_d(\Gamma)$, is the simplicial complex with vertex set $\Gamma$ where a subset $\set{\gamma_0,\cdots,\gamma_n}$ of $\Gamma$ spans a simplex if and only if $d(\gamma_i,\gamma_j)\leq d$ for all $i,j$. Write points $x$ in such a simplex $\sigma_{\set{\gamma_0,\cdots,\gamma_n}}$ of $P_d(\Gamma)$ as formal linear combinations:
$$
x=\sum_{i=0}^n t_i \gamma_i,
$$
where each coefficient $t_i$ is in $[0,1]$, and $\sum_{i=0}^n t_i=1$.  Let $\mathrm{S}(\mathbb{R}^{n+1})$ be the sphere in the Euclidean space $\mathbb{R}^{n+1}$, and define a bijection from the simplex $\sigma_{\set{\gamma_0, \cdots, \gamma_n }}$ to $\mathrm{S}(\mathbb{R}^{n+1})$ via the map
$$
\rho: x=\sum_{i=0}^n t_i \gamma_i \mapsto \Big( \sum_{i=0}^n t_i^2 \Big)^{-\frac{1}{2}} (t_0, \cdots, t_n).
$$
The spherical metric on $\sigma\set{\gamma_0, \cdots, \gamma_n }$ is the metric defined by
$$
d_\sigma(x, y):=\frac{2}{\pi} \arccos \Big( \Big \langle \rho(x), \rho(y) \Big \rangle \Big),
$$
i.e. the length (normalized by $2/\pi$) of the shorter arc of a great circle connecting $\rho(x)$ and $\rho(y)$.

For points $x, y\in P_d(\Gamma)$, a {\em simplicial path} $\gamma$ (cf. \cite{WY}) between them is a finite sequence $x=x_0, x_1, \cdots, x_n=y$ of points
in $P_d(\Gamma)$ together with a choice of simplices $\sigma_1, \cdots, \sigma_n$ such that each $\sigma_i$ contains $(x_{i-1}, x_i)$. The length of
such a path $\gamma$ is defined to be
$$
l(\gamma):=\sum_{i=1}^n d_{\sigma_i}(x_{i-1}, x_i),
$$
and the {\em spherical distance} between two arbitrary points $x, y\in P_d(\Gamma)$ is defined to be
$$
d_{S}(x, y):=\inf \Big\{ l(\gamma): \gamma \mbox{ a simplicial path between } x \mbox{ and } y \Big\}
$$
and $d_S(x, y)=\infty$ if no simplicial path exists.

A {\em semi-simplicial path} $\delta$  (see Definition 7.2.8 in \cite{WY}) between points $x$ and $y$ in $P_d(\Gamma)$ consists of a sequence of the form
$$x=x_0, y_0, x_1, y_1, x_2, y_2, \cdots, x_n, y_n=y,$$
where each of the points $x_i$ and $y_i$ are in $\Gamma$ and {\em some of these points may be repeated}. The {\em length} of such a path is defined as
$$
\ell(\delta):=\sum_{i=0}^n d_S(x_i, y_i) + \sum_{i=0}^{n-1} d_{\Gamma} (y_i, x_{i+1}).
$$
The {\em semi-spherical distance} on $P_d(\Gamma)$ is defined by
$$
d_{P_d}(x, y):= \inf \Big\{ \ell(\delta) \Big| \; \delta \mbox{ a semi-simplicial path between } x \mbox{ and } y \Big\}
$$
Note that a semi-simplicial path between two points always exists.

The {\em Rips complex of $\Gamma$} is defined to be the space $P_d(\Gamma)$ equipped with the metric $d_{P_d}$ above.

\end{defn}

It turns out that (see \cite{WY}, Proposition 7.2.11) (1) $P_0(\Gamma)$ identifies isometrically with $\Gamma$; (2) for any $d\geq 0$ the Rips complex $P_d(\Gamma)$ is a proper, second
countable metric space; (3) for each $d'\geq d\geq 0$ the canonical inclusion $i_{d'd}: P_d(\Gamma)\to P_{d'}(\Gamma)$ is a coarse equivalence, and a homeoporphism onto its image.

To define the assembly map, we recall that when $p=2$, Yu in \cite{Yu97} proved that the local index map from $K$-homology to $K$-theory of localization algebra is an isomorphism for a finite-dimensional simplicial complex. Y. Qiao and J. Roe in \cite{QR} later generalize this isomorphism to general locally compact metric spaces. Therefore, for $p\in(1,\infty)$, considering the analogs of $\lp$-Roe algebra and $\lp$-localization algebra, we define the evaluation at zero map as the assembly map, which is equivalent to the original index map when $p=2$. The following conjecture is called {\bf the coarse geometric $\ell^p$-Novikov conjecture}:

\begin{conj} If $\Gamma$ is a discrete metric space with bounded geometry, then the assembly map
$$
\mu:=e_*: \lim_{d\to \infty} K_*(\clp(P_d(\Gamma))) \to \lim_{d\to\infty} K_*(\cp(P_d(\Gamma)))\cong
K_*(\cp(\Gamma))
$$
is injective, where $1<p<\infty$.
\end{conj}

\vskip 10mm

\section{An $\ell^p$ coarse Mayer-Vietoris principle}
In this section, we present an $\ell^p$ coarse Mayer-Vietoris principle similar to the argument in \cite{HRY}.

\begin{defn}[\cite{HRY}]
Let $X$ be a proper metric space, and let $A$ and $B$ be closed subspaces with $X=A\cup B$. We say that $(A,B)$ is an $\omega$-excisive pair, or that $X=A\cup B$ is an $\omega$-excisive decomposition, if for each $R>0$ there is some $S>0$ such that
$$
\mathrm{Pen}(A;R)\cap \mathrm{Pen}(B;R)\subset \mathrm{Pen}(A\cap B;S),
$$
where $\mathrm{Pen}(A;R)=\{y\in X |\; d(y, A)\leq R\}$ is the $R$-neighborhood of $A$ in $X$.
\end{defn}

\begin{defn}[\cite{HRY}]
Let $A$ be a closed subspace of a proper metric space $X$. Denote by $\cp(A;X)$ the operator-norm closure of the set of all locally compact, finite propagation operators $T$ on $\ell^p(Z, \ell^p))$ whose support is contained in $\mathrm{Pen}(A;R)\times \mathrm{Pen}(A;R)$, for some $R>0$ depending on $T$.
\end{defn}

One can see that $\cp(A;X)$ is a two-sided ideal of $\cp(X)$. For $s,t\in [0,\infty)$ with $s<t$, the inclusion $\mathrm{Pen}(A;s)\to \mathrm{Pen}(A;t)$ induces an inclusion map
$$
i_{t,s}:\cp(\mathrm{Pen}(A;s))\to \cp(\mathrm{Pen}(A;t)).
$$
It follows that $\cp(A;X)=\lim_{n\to\infty} \cp(\mathrm{Pen}(A;n))$ and we get an induced map
$$
i: B^p(A) \to B^p(A; X).
$$

\begin{lem}[\cite{HRY}]
The induced map at $K$-theory level
$$
i_*:K_*(\cp(A))\to K_*(\cp(A;X))
$$
is an isomorphism.
\end{lem}
\begin{proof}
Since the inclusions $A\subset \mathrm{Pen}(A;n)$ and $\mathrm{Pen}(A;n)\subset \mathrm{Pen}(A;n+1)$ are coarse equivalence, the induced maps on $K$-theory are all isomorphisms.
\end{proof}

Let $A$ be a closed subspace of $S$ and consider the ideal $\cp(A;X)$ of $\cp(X)$. Define idempotents $Q: \, X\times X\to \calk_p$ by the formula:
\begin{align*}
& Q(x,y)=0\text{ if } x\neq y, \\
& Q(x,x)=\left(\begin{array}{cc}
                   I_{r(x)} & 0 \\
                   0 & 0 \\
                 \end{array}
               \right),
\end{align*}
where $I_{r(x)}$ is the $r(x)\times r(x)$ identity matrix for some $r(x)\in\bbn$,  and
$$\mathrm{Supp} ( r(x) )\subset \mathrm{Pen}(A;R) $$
for some $R\geq 0$.
We define a partial order on all such idempotents $Q$ by the following: $Q_2\leq Q_1$ if
$
\rank(Q_2(x,x))\leq \rank(Q_1(x,x))$ for all $x\in X.$

Let $\calq$ be the set of all such operators $Q$ with this order.

\begin{prop}\label{approximateunit}
The collection $\calq$ is an approximate unit of $\cp(A;X)$.
\end{prop}
\begin{proof}
Let $T\in \cp(A;X)$ and $\epsilon>0$. For any $x,y\in X$, $T(x,y)$ is either a zero operator or a compact operator on $\ell^p$. Recall that $Z\subseteq X$ is a chosen countable dense subset in the definition of the $\ell^p$ Roe algebra. Enumerate $Z\times Z$ such that each pair
$(x,y)\in Z\times Z$  is assigned a unique integer $n\in \mathbb{N}$.

Let $F(x,y)$ be a finite rank operator on $\ell^p$ such that $\norm{T(x,y)-F(x,y)}<\frac{1}{2^n}\epsilon$ when $T(x,y)\neq 0$ and $n$ is the corresponding integer of $(x,y)$, and $F(x,y)=0$ when $T(x,y)=0$. Then $F=\left(F(x,y)\right)$ is a locally  finite rank operator of finite propagation with $\norm{T-F}<\epsilon$.

For each fixed $x$, since $F$ has finite propagation, there are only finitely many $y$ such that $F(x,y)\neq 0$. Let
$$Q(x,x)=\left(\begin{array}{cc}
                   I_{r(x)} & 0 \\
                   0 & 0 \\
                 \end{array}\right)
$$
be a finite-rank projection for some $r(x)\in\bbn$ such that
$$Q(x,x)F(x,y)=F(x,y)$$ for all $y$ with $F(x,y)\neq 0$.
Define $Q(x,y)=0$ if $x\neq y$. Then
$$
\norm{Q T-T}\leq \norm{Q T- Q F}+\norm{Q F-F}+\norm{F-T}.
$$
Here
$\norm{F-T}< \epsilon$, $Q F-F=0$ and $\norm{Q T-Q F}\leq\norm{Q}\norm{F-T}<\epsilon$. So $\norm{Q T-T}< 2\epsilon$ and the proof is done.
\end{proof}

\begin{prop} \label{twoideals} Let $A, B$ be closed subspaces of $X$ such that $X=A\cup B$. Then
\begin{enumerate}
\item $\cp(A;X)+\cp(B;X)=\cp(X)$;
\item $\cp(A;X)\cap \cp(B;X)=\cp(A\cap B;X)$ if $X=A \cup B$ is an $\omega$-excisive decomposition.
\end{enumerate}
\end{prop}
\begin{proof}
Obviously, $\cp(A;X)+\cp(B;X)\subseteq\cp(X)$. For the reverse inclusion, let $\chi_A$ be the characteristic function of $A$. For any $T\in\cp(X)$ and $\epsilon>0$, there exists $T_\epsilon\in \calgp(X)$ with $\norm{T-T_\epsilon}<\epsilon$. Then $T_\epsilon\chi_A\in\cp(A;X)$ and $\norm{T\chi_A-T_\epsilon\chi_A}\leq \norm{T-T_\epsilon}<\epsilon$. It follows that $T\chi_A\in\cp(A;X)$ and consequently, $T=T\chi_A+T(1-\chi_A)\in\cp(A;X)+\cp(B;X)$. Therefore $\cp(A;X)+\cp(B;X) \supseteq\cp(X)$.

For the second part, we will show that
$$
B^p(A; X)\cap B^p(B; X)=\overline{B^p(A;X)B^p(B;X)}=B^p(A\cap B; X).
$$
for an $\omega$-excisive pair $(A, B)$ of $X$.

Obviously, $\cp(A\cap B;X)\subseteq \cp(A;X)\cap \cp(B;X)$ holds for any decomposition pair $(A,B)$. On the other hand, by Proposition \ref{approximateunit}, one can easily see that $\cp(A;X)\cap \cp(B;X)\subseteq \overline{\cp(A;X)\cp(B;X)}$. Finally, for $T_A\in\calgp(A;X)$ and $T_B\in \calgp(B;X)$ with
\begin{align*}
\mathrm{Supp}(T_A)&\subseteq \mathrm{Pen}(A;R)\times \mathrm{Pen}(A;R);\\
\mathrm{Supp}(T_B)&\subseteq \mathrm{Pen}(B;R')\times \mathrm{Pen}(B;R').
\end{align*}
since $(A,B)$ is $w$-excisive, there exists $S>0$ such that
$$
\mathrm{Supp}(T_A T_B)\subseteq \mathrm{Pen}(A\cap B;S)\times \mathrm{Pen}(A\cap B;S).
$$
Hence $\overline{\cp(A;X)\cp(B;X)}\subseteq \cp(A\cap B;X)$. This completes the proof.
\end{proof}
As a general fact (cf. Proof of Proposition 2.7.15 in \cite{WY}), if $A$ is a Banach algebra, and $I$ and $J$ are two closed two-sided ideals of $A$ such that $I+J=A$, then standard isomorphism theorems in pure algebra give that
$$
\frac{I}{I\cap J} \cong \frac{I+J}{J} = \frac{A}{J},
$$
which further induces the following Mayer-Vietoris exact sequence (cf. Proposition 2.7.15 in \cite{WY}).
\begin{prop} \label{Mayer-Vietoris} (cf. \cite{WY})
Let $A$ be a Banach algebra, and let $I$ and $J$ be two closed two-sided ideals of $A$ such that $I+J=A$. Then there is a six term Mayer-Vietoris exact sequence on $K$-theory
\[
\xymatrix{
K_0(I\cap J) \ar[r] & K_0(I) \oplus K_0(J) \ar[r] & K_0(A) \ar[d] \\
K_1(A) \ar[u] & \ar[l] K_1(I) \oplus K_1(J) & \ar[l] K_1(I \cap J).
}
\]
\end{prop}

Combining these lemmas, we obtain the following $\lp$-version of the coarse Mayer-Vietoris principle.
\begin{prop}
Let $X=A\cup B$ be an $\omega$-excisive decomposition of $X$. Then there is a six term Mayer-Vietoris exact sequence:
\[
\xymatrix{
K_0(\cp(A\cap  B)) \ar[r] & K_0 (\cp(A))\oplus K_0(\cp(B)) \ar[r] &  K_0(\cp(X)) \ar[d] \\
K_1(\cp(X)) \ar[u] &  \ar[l]  K_1 (\cp(A))\oplus K_1(\cp(B)) & \ar[l]  K_1(\cp(A\cap  B)).
}
\]
\end{prop}

\vskip 15mm

\section{Twisted $\lp$-Roe algebras and twisted $\lp$-localization algebras}
In this section, we shall define the twisted $\lp$-Roe algebras and the twisted $\lp$-localization algebras for bounded
geometry spaces which admit a coarse embedding into a simply connected complete Riemannian manifold of
nonpositive sectional curvature. The construction of these twisted $\lp$-algebras follows those twisted algebras
introduced in \cite{Yu00}, with technical adjustments suitable to $\ell^p$ spaces.

Let $M$ be a simply connected complete Riemannian manifold of nonpositive sectional curvature. In the following,
we shall assume that the dimension of $M$ is even. If $\dim(M)$ is odd, we can replace $M$ by $M\times \bbr$.
Indeed, the product manifold $M\times \bbr$ is also a simply connected complete Riemannian manifold
with nonpositive sectional curvature. And if $f: \Gamma\to M$ is a coarse embedding, then the induced map
$f': \Gamma\to M\times \bbr$ defined by $f'(\gamma)=(f(\gamma), 0)$ is also a coarse embedding so that
we can replace $f$ by $f'$.

Let $\cA=C_0(M,\cliff_{\bbc}(TM))$ be the $C^*$-algebra of continuous sections $a$ on $M$ which have value $a(x)\in \cliff_\bbc(T_xM)$ at each point $x\in M$ and vanish at infinity, where $\cliff_\bbc(T_xM)$ is the complexified Clifford algebra \cite{ABS} of the tangent space $T_xM$ at $x\in M$ with respect to the inner product on $T_xM$ given by the Riemannian structure of $M$. Here $\cliff_{\bbc}(TM)$ is the Clifford algebra bundle over $M$. Meanwhile, for any $x\in M$, $\cliff_\bbc(T_xM)$ is also a Hilbert space, so that $\cliff_\bbc(TM)$ is a Hilbert space bundle. Let $$\calb:=L^p(M,\cliff_{\bbc}(TM)),$$
the set of all $L^p$ sections of the Hilbert space bundle $\cliff_\bbc(TM)$, which is a Banach space. The $C^*$-algebra $\cala$ acts on $\calb$ by pointwise multiplications, so that it can be regarded as an $L^p$-operator algebra (cf. \cite{Phillips12,NCP}).

Let's make this point of view more precisely. Let $\nu$ be the Radon measure on $M$ induced by the Riemannian metric on $M$ (cf. \cite{Lang}, Chapter XVI, Theorem 4.4). The continuous sections with compact support of the Hilbert space bundle $\cliff_\bbc(TM)$ generates a local $\nu$-measurability structure $\mathcal{W}$ for the cross sections of  $\cliff_\bbc(TM)$ (cf. \cite{Fell-Doran}, Chapter II, section 15). For $1\leq p<\infty$, the Banach space $\mathfrak{B}=L^p(M,\cliff_{\bbc}(TM))$ consists of all those locally $\mu$-measurable cross-sections $f$ of $\mathcal{W}$ such that
$$\|f\|_p^p=\int_M \|f(x)\|^p d \nu (x) \leq \infty,$$
and two elements of $L^p(M,\cliff_{\bbc}(TM))$ are identical if they differ only on a $\mu$-null set (cf. \cite{Fell-Doran}, Chapter II, section 15.7).
Since $M$ is a simply connected complete Riemannian manifold with nonpositive sectional curvature, by the Cartan-Hadamard theorem, for any $x\in M$, the exponential map
$\exp_x: T_x M\to M$ gives rise to a diffeomorphism from $\mathbb{R}^n$ to $M$, so that the Hilbert space bundle  $\cliff_\bbc(TM)$ is isomorphic to the trivial bundle $M\times \mathcal{M}_{2^k}(\mathbb{C})$, where $n=2k=\dim(M)$ and the matrix algebra $\mathcal{M}_{2^k}(\mathbb{C})$ is endowed with a Hilbert space structure induced from the Hilbert space structure of $\cliff_\bbc(T_x M)$. Consequently,  we have
$$
\begin{array}{rcl}
\mathfrak{B} & := & L^p(M,\cliff_{\bbc}(TM)) \\
 & \cong &  L^p(M, \nu; \mathcal{M}_{2^k}(\mathbb{C}))  \\
 & \cong &  L^p(M, \nu)\otimes_p \mathcal{M}_{2^k}(\mathbb{C}).
\end{array}
$$

\par
Let us recall some facts about the tensor norms on the spaces of $p$-integrable functions and tensor product operators (cf. \cite{DefantFloret}, Chapter 7; and \cite{Figiel-Iwaniec-Pelczynski}, Theorem 1.1 and Corollary 1.1. For a good summary, see \cite{Chung} or \cite{Phillips12}). Let $(\Omega, \mu)$ be an arbitrary measure space, $1\leq p<\infty$, and $E$ a Banach space. Then the space $L^p(\mu, E)$ of (classes of a.e. equal) Bochner $p$-integrable functions provides the algebraic tensor product $L^p(\mu)\otimes_{alg} E$ with a ``natural tensor norm" $\Delta_p$ via the natural injective mapping
\[
\begin{array}{ccc}
L^p(\mu) \otimes_{alg} E & \hookrightarrow & L^p(\mu, E) \\
f \otimes x & \mapsto & f(\cdot) x
\end{array}
\]
The completion of $L^p(\mu) \otimes_{alg} E$ against the tensor norm $\Delta_p$ is denoted in the following by $L^p(\mu)\otimes_p E$. Since the simple functions are dense in $L^p(\mu, E)$, the above inclusion induces an isometric isomorphism
$$L^p(\mu)\otimes_p E \cong L^p(\mu, E).$$

For the product measure $\mu\times \nu$ on $\Omega_1\times \Omega_2$, the Fubini-Tonelli theorem shows that the inclusion
$L^p(\mu)\otimes_{alg} L^p(\nu) \hookrightarrow L^p(\mu\times \nu)$ induces  isometric isomorphisms
\[
L^p(\mu)\otimes_p L^p(\nu) \cong L^p(\mu, L^p(\nu)) \cong L^p(\mu\times \nu).
\]
Moreover, the Fubini theorem for Bochner integrals  (cf. \cite{Williams}, Appendices, Theorem B.41; or more generally,  \cite{Fell-Doran}, Chapter II, section 16) shows that we may replace the space $L^p(\nu)$ in the above identifications by $L^p(\nu, F)$ of Bochner $p$-integrable functions into a Banach space $F$. In particular, we have
\[
\begin{array}{rcl}
L^p(\mu)\otimes_p \big( L^p(\nu)\otimes_p F \big) & \cong & L^p(\mu)\otimes_p L^p(\nu, F)  \\
&  \cong & L^p(\mu, L^p(\nu, F))  \\
& \cong &  L^p(\mu\times \nu, F).
\end{array}
\]
provided that the spaces $\Omega_1$ and $\Omega_2$ are locally compact and $\sigma$-compact Hausdorff spaces.

In general, for bounded linear operators $S\in \mathcal{L}(L^p(\mu))$ and $T\in \mathcal{L}(E)$, there are natural examples showing that
the tensor product operator
$$S\otimes T: L^p(\mu)\otimes_{alg} E \longrightarrow L^p(\mu)\otimes_{alg} E$$
may not extend to a bounded operator on $L^p(\mu)\otimes_{p} E$ (cf. \cite{DefantFloret}, Chapter 7, section 7.5 and 7.6).
However, in the situation considered in this paper, we do not meet this difficulties. Namely, on one hand,
if $S=id_{L^p(\mu)}$, it is easy to verify that
$$\|id\otimes T: L^p(\mu) \otimes_p E \longrightarrow L^p(\mu)\otimes_p E\|= \|T\|$$
for any $T\in \mathcal{L}(E)$ (cf. the proof of Theorem 1.2 in \cite{Figiel-Iwaniec-Pelczynski}). On the other hand, if $E=L^p(\nu, F)$ for a Banach space $F$ and the same $p$ as in $L^p(\mu)$, and if $id_{L^p(\nu, F)}$ is the identity operator, then for any bounded operator $S\in \mathcal{L}(L^p(\mu))$, the operator
$$S\otimes id_{L^p(\nu, F)}: L^p(\mu)\otimes_{alg} L^p(\nu, F) \longrightarrow L^p(\mu)\otimes_{alg} L^p(\nu, F)$$
has a unique extension to a bounded linear operator
$$S\otimes id_{L^p(\nu, F)}: L^p(\mu)\otimes_{p} L^p(\nu, F) \longrightarrow L^p(\mu)\otimes_{p} L^p(\nu, F)$$
such that
$$ \|S\otimes id_{L^p(\nu, F)}\|=\|S\|.$$
This can be proved by appealing to the proof of Theorem 1.1 in \cite{Figiel-Iwaniec-Pelczynski} (replacing $L^p(\sigma)$ there by $L^p(\nu, F)$ here and confining to the case $p=q$, so that the integral version of the Minkowski inequality for $\alpha=1$ still holds).
Consequently, for any bounded linear operators $S\in \mathcal{L}(L^p(\mu))$ and $T\in \mathcal{L}(L^p(\nu, F))$, the operator
$S\otimes T=(S\otimes id)(id\otimes T)$ on $\mathcal{L}(L^p(\mu)) \otimes_{alg} \mathcal{L}(L^p(\nu, F))$ extends continuously to a unique bounded linear operator
$$S\otimes T: \mathcal{L}(L^p(\mu)) \otimes_{p} \mathcal{L}(L^p(\nu, F)) \longrightarrow  \mathcal{L}(L^p(\mu)) \otimes_{p} \mathcal{L}(L^p(\nu, F))$$
and $\|S\otimes T\|=\|S\| \|T\|$ (cf. the proof of Corollary 1.1 in \cite{Figiel-Iwaniec-Pelczynski}).

In this paper, we will only concern ourselves with the situation where $L^p(\mu)=\ell^p$ and
$$L^p(\nu, F)=\calb:=L^p(M,\cliff_{\bbc}(TM))\cong L^p(M, \nu; \mathcal{M}_{2^k}(\mathbb{C})).$$
For $a\in \cA=C_0(M,\cliff_{\bbc}(TM))$ and $h\in\calb$. Define
$$\|a\|_{\infty}=\sup\set{\norm{a(x)} \; | \; x\in M}.$$
Then $\norm{a\cdot h}\leq \|a\|_{\infty}\norm{h}$ and $\cala\subset \call(\calb)$. For $n\in\bbn$, define
$$\calb_{n,p}=\calb\oplus_p\cdots\oplus_p \calb,$$
the $\lp$-direct sum of $n$ copies of $\calb$. The $\lp$-norm of $\calb_{n,p}$ is defined as
$$
\norm{(f_1,\cdots,f_n)}_{p}=\Big( {\sum^n_{i=1}\norm{f_i}^p} \Big)^{1/p},\quad\quad \text{for }f_1,\cdots,f_n\in \calb.
$$
Let $M_n(\cala)$ be the  algebra of $n\times n$ matrices with entries in $\cala$. Then $M_n(\cala)$ acts on $\calb_{n,p}$ by matrix multiplications, so that $M_n(\cala)\subset \call(\calb_{n,p})$.
Embed $M_n(\cala)$ into $M_{n+1}(\cala)$ at the top left corner, and let $M_{\infty,p}(\cala)$ be the inductive limit of $\set{M_{n}(\cala)}_{n=1}^\infty$. Define
$$\calb_{\infty,p}=\calb\oplus_p\cdots\oplus_p \calb \oplus_p \cdots$$
to be the $\lp$-direct sum of infinitely many copies of $\calb$ with the $\lp$-norm
$$
\norm{\set{f_i}^\infty_{i=1}}_p=\Big( {\sum^\infty_{i=1}\norm{f_i}^p} \Big)^{1/p},\quad\quad \text{for }\set{f_i}^\infty_{i=1}\in\calb_{\infty,p}.
$$
It follows from the above discussions that we have isometric isomorphisms
$$\calb_{\infty,p}\cong \lp(\bbn,\calb)\cong\lp\otimes_p \calb$$
and all $M_{n}(\cala)$ can be considered as subalgebras of $\call(\calb_{\infty,p})$. Denote by $\calk_p\otimes_{alg}\cala$ the algebraic tensor product of $\calk_p$ and $\cala$. Naturally $\calk_p\otimes_{alg}\cala$ acts on $\calb_{\infty,p}$ and $\calk_p\otimes_{alg}\cala\subset\call(\calb_{\infty,p})$. Let
$$\calk_p\otimes_p \cala=\overline{\calk_p\otimes_{alg}\cala}^{\call(\calb_{\infty,p})}.$$
It follows that $\calk_p\otimes_p \cala\cong M_{\infty, p}(\cala)$.
\par
Let $\Gamma$ be a discrete metric space with bounded geometry. Let $f:\; \Gamma\to M$ be a coarse embedding.
For each $d> 0$, we shall extend the map $f$ to the Rips complex $P_d(\Gamma)$ in the following way.
Note that $f$ is a coarse map, i.e., there exists $R>0$ such that for all $\gamma_1,\gamma_2\in\Gamma$,
$$
d(\gamma_1,\gamma_2)\leq d \; \Longrightarrow  \; d_M(f(\gamma_1),f(\gamma_2))\leq R.
$$
For any point $x=\sum_{\gamma\in \Gamma} c_\gamma \gamma \in P_d(\Gamma)$, where  $c_\gamma\geq 0$ and
$\sum_{\gamma\in \Gamma} c_\gamma=1$,  we  choose a point $f_x\in M$ such that
$$
d(f_x, f(\gamma))\leq R
$$
for all $\gamma\in \Gamma$ with $c_\gamma\not=0$. The correspondence $x\mapsto f_x$ gives a coarse embedding
$P_d(\Gamma)\to M$, also denoted by $f$.

Choose a countable dense subset $\Gamma_d$ of $P_d(\Gamma)$ for each $d>0$ in such a way that
$\Gamma_d\subset \Gamma_{d'}$ when $d<d'$.

\begin{defn}\label{twistedRoealgebra}
Let $\calgp(P_d(\Gamma),\cA)$ be the set of all functions
$$
T:\Gamma_d\times \Gamma_d\to \cK_p\otimes_p\cA\subset\call(\calb_{\infty,p})=\call(\lp\otimes_p L^p(M,\cliff_\bbc(TM)))
$$
such that
\begin{enumerate}
\item there exists $C>0$ such that $\|T(x,y)\| \leq C$ for all $x,y\in \Gamma_d$;
\item there exists $R>0$ such that $T(x,y)=0$ if $d(x,y)>R$;
\item there exists $L>0$ such that for every $z\in P_d(\Gamma)$, the number of elements
in the following set
$$\{ (x, y)\in\Gamma_d\times \Gamma_d: \; d(x, z)\leq 3R, \; d(y, z)\leq 3R, \; T(x,y)\neq 0 \}$$
is less than $L$.
\item there exists $r>0$ such that
$$
\supp(T(x,y))\subset B(f(x),r)
$$
for all $x,y\in \Gamma_d$,  where $B(f(x),r)=\{m\in M: \; d(m,f(x))<r\}$ and, for all $x,y\in\Gamma_d$, the entry
$T(x,y)\in \cK_p\otimes_p\cA$ is a function on $M$ with $T(x, y)(m)\in \cK_p\otimes_p\Cliff_{\bbc}(T_mM)$ for each $m\in M$
so that the {\em support} of $T(x,y)$ is defined by
$$
\supp(T(x,y)):= \{m\in M: \; T(x,y)(m)\neq 0\}.
$$
\end{enumerate}
\end{defn}

For $f\in \ell^p(\Gamma_d, \calb_{\infty,p})$, we define
$$
Tf(x)=\sum_{y\in\Gamma_d} T(x,y)f(y).
$$
Then $T=(T(x,y))\in \call(\ell^p(\Gamma_d,\calbp))$.

\begin{defn} The twisted $\lp$-Roe algebra $\cp(P_d(\Gamma),\cA)$ is defined to be the operator norm closure
of $\calgp(P_d(\Gamma), \cA)$ in $\call(\ell^p(\Gamma_d,\calbp))$.
\end{defn}

The above definition of the twisted $\lp$-Roe algebra is similar to that in \cite{Yu00}.

Let $\clalgp(P_d(\Gamma),\cA)$ be the set of all bounded, uniformly norm-continuous  functions
$$
g:\; \mathbb{R}_+\to \calgp(P_d(\Gamma),\cA)
$$
such that
\begin{enumerate}
\item there exists a bounded function $R(t): \;\mathbb{R}_+\to\mathbb{R}_+$ with
                $\displaystyle \lim_{t\to\infty} R(t)=0$ such that $(g(t))(x,y)=0$ whenever $d(x,y)>R(t)$;
\item there exists $L>0$ such that for every $z\in P_d(\Gamma)$, the number of elements
                    in the following set
                    $$\{ (x, y)\in\Gamma_d\times \Gamma_d: \; d(x, z)\leq 3R, \; d(y, z)\leq 3R, \; g(t)(x,y)\neq 0 \}$$
                    is less than $L$ for every $t\in \mathbb{R}_+$.
\item there exists $r>0$ such that $\supp((g(t))(x,y))\subset B(f(x),r)$ for all
                $t\in\mathbb{R}_+$, $x,y\in\Gamma_d$, where $f:P_d(\Gamma)\to M$ is the extension of the coarse
                embedding $f:\Gamma\to M$ and $B(f(x),r)=\{m\in M: d(m,f(x))<r\}$.
\end{enumerate}

\begin{defn} The twisted $\lp$-localization algebra $\clp(P_d(\Gamma),\cA)$ is defined to be the norm
completion of $\clalgp(\pdg,\cA)$, where $\clalgp(\pdg,\cA)$ is endowed with the norm
$$
\|g\|_\infty=\sup_{t\in\mathbb{R}_+} \|{g(t)}\|_{\cp(\pdg,\cA)}.
$$
\end{defn}

The above definition of the twisted $\lp$-localization Roe algebra is similar to that in \cite{Yu00}. The evaluation homomorphism $e$ from $\clp(\pdg,\cA)$ to $\cp(\pdg,\cA)$ defined by $e(g)=g(0)$ induces a homomorphism
at $K$-theory level:
$$
e_*: \lim_{d\to\infty}K_*(\clp(\pdg,\cA))\to\lim_{d\to\infty}K_*(\cp(\pdg,\cA)).
$$

\begin{thm}\label{twisted-iso}
Let $\Gamma$ be a discrete metric space with bounded geometry which admits a coarse embedding
$f: \Gamma\to M$ into a simply connected, complete Riemannian manifold $M$ of non-positive sectional curvature.
Then the homomorphism
$$
e_*: \lim_{d\to\infty}K_*(\clp(\pdg,\cA))\to\lim_{d\to\infty}K_*(\cp(\pdg,\cA)).
$$
is an isomorphism.
\end{thm}
The proof of Theorem \ref{twisted-iso} will follow the proof of Theorem 6.8 in \cite{Yu00}. To begin with, we
need to discuss ideals of the twisted algebras associated to open subsets of the manifold $M$.

\begin{defn}\hfill
\begin{enumerate}
\item The $support$ of an element $T$ in $\calgp(\pdg,\cA)$ is defined to be
$$
\begin{array}{ccl}
\supp(T)  & = & \Big\{(x,y,m)\in\Gamma_d\times\Gamma_d\times M:\; m\in\supp(T(x,y))\Big\}    \\
          & = & \Big\{(x,y,m)\in\Gamma_d\times\Gamma_d\times M:\; (T(x,y))(m)\not=0 \Big\};
\end{array}
$$
\item The $support$ of an element $g$ in $\clalgp(\pdg,\cA)$ is defined to be
$$
\bigcup_{t\in\mathbb{R}_+} \supp(g(t)).
$$
\end{enumerate}
\end{defn}

Let $O\subset M$ be an open subset of $M$. Define $\calgp(\pdg,\cA)_O$ to be the subalgebra of
$\calgp(\pdg,\cA)$ consisting of all elements whose supports are contained in $\Gamma_d\times\Gamma_d\times O$,
i.e.,
$$
\calgp(\pdg,\cA)_O \; = \; \{T\in \calgp(\pdg,\cA): \; \supp(T(x,y))\subset O,\; \forall \; x,y\in\Gamma_d \}.
$$
Define $\cp(\pdg,\cA)_O$ to be the norm closure of $\calgp(\pdg,\cA)_O$.  Similarly, let
$$
\clalgp(\pdg,\cA)_O=\Big\{ g\in \clalgp(\pdg,\cA): \; \supp(g)\subset\Gamma_d\times\Gamma_d\times O \Big\}
$$
and define $\clp(\pdg,\cA)_O$ to be the norm closure of $\clalgp(\pdg,\cA)_O$ under the norm
$\|g\|_\infty=\sup_{t\in\mathbb{R}_+} \|{g(t)}\|_{\cp(\pdg,\cA)}$.

Note that $\cp(\pdg,\cA)_O$ and $\clp(\pdg,\cA)_O$ are closed two-sided ideals of $\cp(\pdg,\cA)$ and
$\clp(\pdg,\cA)$, respectively. We also have an evaluation homomorphism
$$e: \clp(\pdg,\cA)_O \to \cp(\pdg,\cA)_O$$
given by $e(g)=g(0)$.

\begin{lem}\label{decomposition}
For any two open subsets $O_1, O_2$ of $M$, we have
$$\cspace_{O_1}+\cspace_{O_2}=\cspace_{O_1\cup O_2},$$
$$\cspace_{O_1}\cap\cspace_{O_2}=\cspace_{O_1\cap O_2},$$
$$\clspace_{O_1}+\clspace_{O_2}=\clspace_{O_1\cup O_2},$$
$$\clspace_{O_1}\cap\clspace_{O_2}=\clspace_{O_1\cap O_2}.$$
Consequently,  we have the following commuting diagram connecting two Mayer-Vietoris sequences at $K$-Theory level:
\[
\xymatrix
{
  & AL_0 \ar[rr] \ar'[d][dd]     &  &  BL_0\ar[rr]\ar'[d][dd]  &  &  CL_0 \ar[dd]^{e_*} \ar[dl]        \\
    CL_1 \ar[ur]\ar[dd]_{e_*}    &  &  BL_1 \ar[ll]\ar[dd]     &  &  AL_1 \ar[ll]\ar[dd]          &    \\
  & A_0 \ar'[r][rr]              &  &  B_0 \ar'[r][rr]         &  &  C_0 \ar[dl]                       \\
    C_1 \ar[ur]                  &  &  B_1\ar[ll]              &  &  A_1 \ar[ll]                  &        }
\]
where, for $*=0, 1$,
$$AL_*=K_*\Big(\clspace_{O_1\cap O_2}\Big), \quad CL_*=K_*\Big(\clspace_{O_1\cup O_2}\Big),$$
$$\;\;\; A_*=K_*\Big(\cspace_{O_1\cap O_2}\Big), \quad \;\; C_*=K_*\Big(\cspace_{O_1\cup O_2}\Big),$$
$$BL_*=K_*\Big(\clspace_{O_1}\Big)\bigoplus K_*\Big(\clspace_{O_2}\Big),$$
$$\;\; \; B_*=K_*\Big(\cspace_{O_1}\Big)\bigoplus K_*\Big(\cspace_{O_2}\Big).
$$
\end{lem}
\begin{proof}We shall prove the first two equalities. The other two equalities can be proved similarly. Then the two
Mayer-Vietoris exact sequences follow from Proposition \ref{Mayer-Vietoris}.

To prove the first equality, it suffices to show that
$$
\calgp(\pdg, \cA)_{O_1\cup O_2}\subseteq \calgp(\pdg, \cA)_{O_1}+\calgp(\pdg, \cA)_{O_2}.$$
Now suppose $T\in \calgp(\pdg, \cA)_{O_1\cup O_2}$. Take a continuous partition of unity
$\{\varphi_1, \varphi_2\}$ on $O_1\cup O_2$ subordinate to the open over $\{O_1, O_2\}$ of $O_1\cup O_2$.
Define two functions
$$
T_1, \; T_2: \;\; \Gamma_d\times \Gamma_d \longrightarrow \cK_p\otimes_p \cA
$$
by
$$T_1(x, y)(m)=\varphi_1(m) \Big( T(x, y)(m) \Big),$$
$$T_2(x, y)(m)=\varphi_2(m) \Big( T(x, y)(m) \Big) $$
for $x, y\in \Gamma_d$ and $m\in M$.

Then $T_1\in \calgp(\pdg, \cA)_{O_1}$, $T_2\in \calgp(\pdg, \cA)_{O_2}$, and
$$
T=T_1+T_2\in \calgp(\pdg, \cA)_{O_1}+\calgp(\pdg, \cA)_{O_2}
$$
as desired.

For the second equality, similar to the proof of Proposition \ref{Mayer-Vietoris}, it suffices to show that
$$
\cspace_{O_1}\cap\cspace_{O_2}   \subseteq  \overline{\cspace_{O_1}\cspace_{O_2}}.
$$

Consider all (rank) functions $r: \Gamma_d \to \mathbb{N}$, and all pairs $(K, \phi)$ where $K\subset O_1$ is a compact subset in $O_1$, and $\phi\in \cA$ is such that
$$\mathrm{Supp} (\phi)\subset O_1, \quad \mbox{ and } \quad \phi|_{K}=1.$$
For any triple $(r; K, \phi)$, define an element $Q\in B^p_{alg}(P_d(\Gamma), \cA)_{O_1}$ by the formula
\[
Q(x,x)= \left(
\begin{array}{cc}
I_{r(x)}  &  0  \\
0   &   0
\end{array}
\right)
\otimes \phi
\]
and $Q(x,y)=0$ if $x\neq y$. It is straightforward that all such elements $Q$ constitute an approximate unit $\mathcal{Q}$ of
$\cspace_{O_1}$. Thus the second equality follows in a similar way to the second equality in Proposition \ref{twoideals}.
This completes the proof.
\end{proof}

It would be convenient to introduce the following notion associated with the coarse embedding
$f: \Gamma\to M$.

\begin{defn}\label{gammarseperate} Let $r>0$. A family of open subsets $\{O_i\}_{i\in J}$ of $M$ is said to be
{\em $(\Gamma, r)$-separate} if
\begin{enumerate}
\item $O_i\cap O_j=\emptyset $ if $i\neq j$;
\item there exists $\gamma_i\in \Gamma$ such that $O_i\subseteq B(f(\gamma_i), r)\subset M$ for each $i\in J$.
\end{enumerate}
\end{defn}

\begin{lem}\label{union-iso}
If $\{O_i\}_{i\in J}$ is a family of $(\Gamma, r)$-separate open subsets of $M$, then
$$
e_*: \lim_{d\to\infty}K_*(\clp(\pdg,\cA)_{\sqcup_{i\in J} O_i})
\to\lim_{d\to\infty}K_*(\cp(\pdg,\cA)_{\sqcup_{i\in J} O_i})
$$
is an isomorphism, where ${\sqcup_{i\in J} O_i}$ is the (disjoint) union of $\{O_i\}_{i\in J}$.
\end{lem}

We will prove Lemma \ref{union-iso} in the next section. Granting Lemma \ref{union-iso} for the moment, we are able to prove Theorem \ref{twisted-iso}.

\begin{proof}[Proof of Theorem \ref{twisted-iso}](\cite{Yu00}). For any $r>0$, we define $O_r\subset M$ by
$$ O_r=\bigcup_{\gamma\in\Gamma}B(f(\gamma),r), $$
where $f:\Gamma\to M$ is the coarse embedding and $B(f(\gamma),r)=\{p\in M:  d(p, f(\gamma))<r\}$.
\par
For any $d>0$, if $r<r'$ then $\cspace_{O_r}\subseteq \cspace_{O_{r'}}$ and $\clspace_{O_r}\subseteq \clspace_{O_{r'}}$.
By definition, we have
$$\cspace=\lim_{r\to\infty}\cspace_{O_r},$$
$$\clspace=\lim_{r\to\infty}\clspace_{O_r}.$$
\par
On the other hand, for any $r>0$, if $d<d'$ then $\Gamma_d\subseteq \Gamma_{d'}$ in
$\pdg\subseteq P_{d'}(\Gamma)$ so that we have natural inclusions
$\cspace_{O_r}\subseteq \cp(P_{d'}(\Gamma), \cA)_{O_r}$ and
$\clspace_{O_r}\subseteq \clp(P_{d'}(\Gamma), \cA)_{O_r}$.
These inclusions induce the following commuting diagram
\[
\scriptsize
\xymatrix
{
  &  K_*(\clp(P_{d'}(\Gamma), \cA)_{O_r}) \ar[rr]^{e_*} \ar'[d][dd]
      &  & K_*(\cp(P_{d'}(\Gamma), \cA)_{O_r})  \ar[dd]        \\
  K_*(\clspace_{O_r})  \ar[ur]\ar[rr]^{\mbox{\qquad\qquad $e_*$}} \ar[dd]
      &  & K_*(\cspace_{O_r}) \ar[ur]\ar[dd] \\
  &  K_*(\clp(P_{d'}(\Gamma), \cA)_{O_{r'}}) \ar'[r]^{\mbox{\qquad\qquad $e_*$}}[rr]
      &  & K_*(\cp(P_{d'}(\Gamma), \cA)_{O_{r'}})                \\
  K_*(\clspace_{O_{r'}})  \ar[rr]^{e_*}\ar[ur]
      &  & K_*(\cspace_{O_{r'}}) \ar[ur]        }
\]
\\
which allows us to change the order of limits from $\displaystyle\lim_{d\to\infty}\lim_{r\to\infty}$ to
$\displaystyle\lim_{r\to\infty}\lim_{d\to\infty}$ in the second piece of the following commuting diagram
\[
\xymatrix{
  \displaystyle\lim_{d\to\infty}K_*(\clspace)      \ar[d]_{\cong} \ar[r]^{e_*}        &
        \displaystyle\lim_{d\to\infty}K_*(\cspace)       \ar[d]_{\cong}                     \\
  \displaystyle\lim_{d\to\infty}\lim_{r\to\infty}K_*(\clspace_{O_r})      \ar[d]_{\cong}     \ar[r]^{e_*}     &
        \displaystyle\lim_{d\to\infty}\lim_{r\to\infty}K_*(\cspace_{O_r}) \ar[d]_{\cong}    \\
  \displaystyle\lim_{r\to\infty}\lim_{d\to\infty}K_*(\clspace_{O_r})      \ar[r]^{e_*}                        &
        \displaystyle\lim_{r\to\infty}\lim_{d\to\infty}K_*(\cspace_{O_r})   }
\]
So, to prove Theorem \ref{twisted-iso}, it suffices to show that, for any $r>0$,
$$
e_*:\; \lim_{d\to\infty}K_*(\clspace_{O_r})\to\lim_{d\to\infty}K_*(\cspace_{O_r})
$$
is an isomorphism.

Let $r>0$. Since $\Gamma$ has bounded geometry and $f:\Gamma\to M$ is a
coarse embedding, there exist finitely many mutually disjoint subsets of $\Gamma$,
say $\Gamma_k:=\{\gamma_i: i\in J_k\}$ with some index set $J_k$ for $k=1, 2, \cdots, k_0$,
such that $\Gamma=\bigsqcup_{k=1}^{k_0} \Gamma_k$ and, for each $k$,
$d(f(\gamma_i), f(\gamma_j))>2r$ for distinct elements $\gamma_i, \gamma_j$ in $\Gamma_k$.

For each $k=1, 2, \cdots, k_0$, let
$$
O_{r,k}=\bigcup_{i\in J_k} B(f(\gamma_i),r).
$$
Then $O_r=\bigcup^{k_0}_{k=1}O_{r,k}$ and each $O_{r,k}$, or an intersection of several $O_{r,k}$,
is the union of a family of $(\Gamma, r)$-separate (Definition \ref{gammarseperate}) open subsets of $M$.

Now Theorem \ref{twisted-iso} follows from Lemma \ref{union-iso} together with a Mayer-Vietoris sequence argument by using Lemma \ref{decomposition}.
\end{proof}

\vskip 5mm

\section{Strong Lipschitz homotopy invariance}
In this section, we shall present Yu's arguments about strong Lipschitz homotopy invariance for $K$-theory of the
twisted localization algebras \cite{Yu00}, and prove Lemma \ref{union-iso} of the previous section.
\par
Let $f: \Gamma\to M$ be a coarse embedding of a bounded geometry discrete metric space $\Gamma$ into a
simply connected complete Riemannian manifold $M$ of nonpositive sectional curvature, and let $r>0$. Let
$\{O_i\}_{i\in J}$ be a  family of $(\Gamma, r)$-separate open subsets of $M$, i.e.,
\begin{enumerate}
\item $O_i\cap O_j=\emptyset $ if $i\neq j$;
\item there exists $\gamma_i\in \Gamma$ such that
$O_i\subseteq B(f(\gamma_i), r)\subset M$ for each $i\in J$.
\end{enumerate}
For $d>0$, let $X_i$, $i\in J$,  be a family of closed subsets of $\pdg$ such that $\gamma_i\in X_i$ for every $i\in J$
and $\{X_i\}_{i\in J}$ is uniformly bounded in the sense that there exists $r_0>0$ such that
$diameter(X_i)\leq r_0$ for each $i\in J$. In particular, we will consider the following three cases of
$\{X_i\}_{i\in J}$:
\begin{enumerate}
\item $X_i=B_{P_d(\Gamma)}(\gamma_i, R):=\{x\in\pdg:\; d(x,\gamma_i)\leq R\}$, for some common $R>0$ for all $i\in J$;
\item $X_i=\Delta_i$, a simplex in $\pdg$ with $\gamma_i\in \Delta_i$ for each $i\in J$;
\item $X_i=\{\gamma_i\}$ for each $i\in J$.
\end{enumerate}

For each $i\in J$, let $\cA_{O_i}$ be the subalgebra of $\cA=C_0(M, \Cliff_{\bbc}(TM))$ generated by those
functions whose supports are contained in $O_i$.

We define $A(X_i: i\in J)$ to be the closed {\bf subalgebra} of the Banach algebra
$$\left\{ \; {\bigoplus_{i\in J}} T_i   \; \left|\;\;  T_i\in \cp(X_i)\otimes_p \cA_{O_i}, \;
                            \sup_{i\in J} \|T_i\| <\infty \right. \right\}
$$
generated by the elements $ {\bigoplus_{i\in J}} T_i  $ for which conditions (3) and (4) from Definition \ref{twistedRoealgebra} are
satisfied by all operators $T_i$, $i\in J$, viewed as functions
$$T_i: \Big(\Gamma_d \cap X_i \Big) \times  \Big(\Gamma_d \cap X_i \Big) \to \cK_p\otimes_p\cA_{O_i}, $$
{\bf uniformly} (cf. \cite{Willett-Yu-II,Yu00}).

Similarly, we define $A_{L, alg}(X_i:i\in J)$ to be the algebra of bounded, uniformly continuous maps
$$g: [0, \infty) \to A(X_i: i\in J)$$
such that if we write
$$g(t)=\bigoplus_{i\in J} \; g_i(t)$$
then conditions (3) and (4) from Definition \ref{twistedRoealgebra} are satisfied by all operators $g_i(t)$, $i\in J$, $t\in [0,\infty)$,
uniformly, and
there exists a bounded function $c(t)$ on $\mathbb{R}_+$ with $\lim_{t\to \infty} c(t)=0$ such that
$$\Big( g_i(t) \Big)(x, y)=0$$
whenever $d(x, y)>c(t)$ for all $i\in J$, $x, y\in \Gamma_d\cap X_i$ and $t\in [0, \infty)$ (cf. \cite{Willett-Yu-II,Yu00}).

Define $A_{L}(X_i:i\in J)$ to be the completion of $A_{L, alg}(X_i:i\in J)$ for the norm
$$\|g\|=\sup_{t\in [0,\infty)} \|g(t)\|.$$
Note that there is an evaluation-at-zero map
$$e: A_{L}(X_i:i\in J) \to A(X_i: i\in J).$$

For each natural number $s>0$, let $\Delta_i{(s)}$ be the simplex with vertices
$\{\gamma\in\Gamma: \; d(\gamma, \gamma_i)\leq s\}$ in $\pdg$ for $d>s$.
\par
\begin{lem}\label{decomposition}Let  $O=\sqcup_{i\in J} O_i$ be the disjoint union of  a family of $(\Gamma, r)$-separate
open subsets $\{O_i\}_{i\in J}$ of $M$ as above. Then
\begin{enumerate}
\item \qquad $\cspace_O
                       \;\; \cong  \;\;  \displaystyle\lim_{R\to \infty}  \;  A(\{x\in \pdg: d(x,\gamma_i)\leq R\}:i\in J)$;
\item \qquad $\clspace_O
                       \;\;  \cong \;\;  \displaystyle\lim_{R\to \infty} \; A_L(\{x\in \pdg: d(x,\gamma_i)\leq R\}:i\in J)$;
\item $\displaystyle\lim_{d\to\infty}\cspace_O
                       \;\; \cong \;\;  \lim_{s\to \infty}  \;  A(\Delta_i{(s)}:i\in J)$;
\item $\displaystyle\lim_{d\to \infty}\clspace_O
                       \;\;  \cong \;\;  \displaystyle\lim_{s\to\infty}  \;  A_L(\Delta_i{(s)}:i\in J)$.
\end{enumerate}
\end{lem}

\begin{proof} (cf. \cite{Yu00}) Let $\cA_O$ be the subalgebra of $\cA=C_0(M,\Cliff_{\bbc}(TM))$ generated by elements whose supports
are contained in $O$. Let $\calb_O=L^p(O,\Cliff_{\bbc}(TM))$ and let $\calb_{O,\infty,p}$ be the $\lp$-direct sum of infinite copies of $\calb_O$ with the $\lp$-norm
$$
\norm{\{f_i\}^\infty_{i=1}}_p=\Big( {\sum^\infty_{i=1}\norm{f_i}^p} \Big)^{1/p},\quad\quad \text{for }\{f_i\}^\infty_{i=1}\in \calb_{O,\infty,p}.
$$
The algebra $\calk_p\otimes_p\cA_O$ acts on $\calb_{O,\infty,p}$ and the algebra $\cp(\pdg,\cala)_O$ acts on $\lp(\Gamma_d, \calb_{O,\infty,p})$. We have a decomposition
$$
\lp(\Gamma_d, \calb_{O,\infty,p})=\left( \bigoplus_{i\in J} \lp(\Gamma_d, \calb_{O_i,\infty,p}) \right)_p.
$$
Each $T\in \calgp(\pdg,\cA)_O$ has a corresponding decomposition
$$
T=\bigoplus_{i\in J} \;  T_i
$$
such that there exists $R>0$ for which each $T_i$ is supported on
$$
\{(x,y,p): \; p\in O_i,\; x,y\in\Gamma_d, \; d(x,\gamma_i)\leq R, \; d(y,\gamma_i)\leq R\}.
$$
On the other hand, the Banach algebra
$\cp(\{x\in\pdg:d(x,\gamma_i)\leq R\})\otimes_p \cA_{O_i}$ acts on
$$
\ell^p\Big( \{x\in\Gamma_d:d(x,\gamma_i)\leq R\}, \;\; \calb_{O_i,\infty,p} \Big),
$$
so that on $\ell^p(\Gamma_d,\calb_{O_i,\infty,p})$, for each $R>0$,
the algebra
$$A(\{x\in\pdg:d(x,\gamma_i)\leq R\}:i\in J)$$
can be
represented as a subalgebra of $\cspace_O$. In this way,  the decomposition
$T=\oplus_{i\in J}T_i$ induces a Banach algebra isomorphism
$$
\cspace_O\cong\displaystyle\lim_{R\to\infty}A\Big( \{x\in\pdg:d(x,\gamma_i)\leq R \}: i\in J \Big)
$$
as desired in (1). Then (2),(3) and (4) follows straightforwardly from (1).
\end{proof}

Now we turn to recall the notion of strong Lipschitz homotopy \cite{Yu97, Yu98, Yu00}.

Let $\{Y_i\}_{i\in J}$ and $\{X_i\}_{i\in J}$ be two families of uniformly bounded closed subspaces of
$\pdg$ for some $d>0$ with $\gamma_i\in X_i$, $\gamma_i\in Y_i$ for every $i\in J$.
A map $g: \; \bigsqcup_{i\in J}X_i\to \bigsqcup_{i\in J}Y_i$ is said to be $Lipschitz$ if
\begin{enumerate}
\item $g(X_i)\subseteq Y_i$ for each $i\in J$;
\item there exists a constant $c$, independent of $i\in J$, such that
$$ d(g(x),g(y))\leq c \; d(x,y) $$
for all $x,y\in X_i$, $i\in J$.
\end{enumerate}

Let $g_1,g_2$ be two Lipschitz maps from $\bigsqcup_{i\in J}X_i$ to $\bigsqcup_{i\in J}Y_i$.
We say $g_1$ is {\it strongly Lipschitz homotopy} equivalent to $g_2$ if there exists a continuous map
$$
F: \; [0,1]\times (\sqcup_{i\in J}X_i)\to \sqcup_{i\in J}Y_i
$$
such that
\begin{enumerate}
\item $F(0,x)=g_1(x)$, $F(1,x)=g_2(x)$ for all $x\in\sqcup_{i\in J}X_i$;
\item there exists a constant $c$ for which $ d(F(t,x),F(t,y))\leq c \; d(x,y)$
for all $x,y\in X_i$, $t\in [0,1]$, where $i$ is any element in $J$;
\item $F$ is equicontinuous in $t$, i.e., for any $\varepsilon>0$ there exists $\delta>0$ such that
$d(F(t_1,x),F(t_2,x))<\varepsilon $ for all $x\in\sqcup_{i\in J}X_i$ if $|t_1-t_2|<\delta$.
\end{enumerate}

We say $\{X_i\}_{i\in J}$ is {\it strongly Lipschitz homotopy} equivalent to $\{Y_i\}_{i\in J}$ if there
exist Lipschitz maps $g_1:\sqcup_{i\in J}X_i\to \sqcup_{i\in J}Y_i$ and
$g_2:\sqcup_{i\in J}Y_i\to \sqcup_{i\in J}X_i$ such that $g_1g_2$ and $g_2g_1$ are
respectively strongly Lipschitz homotopy equivalent to identity maps.

Define $A_{L,0}(X_i:i\in J)$ to be the subalgebra of $A_L(X_i:i\in J)$ consisting of elements
$\oplus_{i\in J}b_i(t)$ satisfying $b_i(0)=0$ for all $i\in J$.

\begin{lem}[\cite{Yu00}]\label{lipschitz-iso}
If $\{X_i\}_{i\in J}$ is strongly Lipschitz homotopy equivalent to $\{Y_i\}_{i\in J}$
then $K_*(A_{L,0}(X_i:i\in J))$ is isomorphic to $K_*(A_{L,0}(Y_i:i\in J))$.
\end{lem}

Let $e$ be the evaluation homomorphism from $A_L(X_i:i\in J)$  to $A(X_i:i\in J)$ given by
$\oplus_{i\in J} \; g_i(t) \mapsto \oplus_{i\in J}g_i(0)$.

\begin{lem}[\cite{Yu00}]\label{family-iso}
Let $\{\gamma_i\}_{i\in J}$ be as above, i.e.,
$O_i\subseteq B(f(\gamma_i), r)\subset M$ for each $i$. If $\{\Delta_i\}_{i\in
J}$ is a family of simplices in $\pdg$ for some $d>0$ such that
$\gamma_i\in \Delta_i$ for all $i\in J$, then
$$ e_*:K_*(A_L(\Delta_i:i\in J))\to K_*(A(\Delta_i:i\in J)) $$
is an isomorphism.
\end{lem}
\begin{proof} (\cite{Yu00}){\bf .}
Note that $\{\Delta_i\}_{i\in J}$ is strongly Lipschitz homotopy equivalent to
$\{\gamma_i\}_{i\in J}$. By an argument of Eilenberg swindle,  we have
$K_*(A_{L,0}(\{\gamma_i\}:i\in J))=0$. Consequently, Lemma \ref{family-iso} follows from Lemma \ref{lipschitz-iso} and the six term exact sequence of Banach algebra $K$-theory.
\end{proof}

We are now ready to give a proof to Lemma \ref{union-iso} of the previous section.

\begin{proof}[Proof of Lemma \ref{union-iso}, \cite{Yu00}]
By Lemma \ref{decomposition} we have the following commuting diagram
\[
\xymatrix{
  \displaystyle\lim_{d\to\infty}\clspace_{\sqcup_{i\in J} O_i}           \ar[d]_{\cong}          \ar[r]^{e}
            &    \displaystyle\lim_{d\to\infty}\cspace_{\sqcup_{i\in J} O_i}          \ar[d]^{\cong}                   \\
  \displaystyle\lim_{s\to\infty}A_L(\Delta_i{(s)}_i: \; i\in J)           \ar[r]^{e}
            &    \displaystyle\lim_{s\to\infty}A(\Delta_i{(s)}_i: \; i\in J)
     }
\]
which induces the following commuting diagram at $K$-theory level
\[
\xymatrix{
  \displaystyle\lim_{d\to\infty} K_*\Big(\clspace_{\sqcup_{i\in J} O_i}\Big)           \ar[d]_{\cong}          \ar[r]^{e_*}
            &    \displaystyle\lim_{d\to\infty} K_*\Big(\cspace_{\sqcup_{i\in J} O_i}\Big)          \ar[d]^{\cong}                   \\
  \displaystyle\lim_{s\to\infty} K_*\Big(A_L(\Delta_i{(s)}: \; i\in J)\Big)           \ar[r]^{e_*}
            &    \displaystyle\lim_{s\to\infty} K_*\Big(A(\Delta_i{(s)}: \; i\in J)\Big) .
     }
\]
Now Lemma \ref{union-iso} follows from Lemma \ref{family-iso}.
\end{proof}

\vskip 5mm

\section{Almost flat Bott elements and Bott  maps}
In this section, we shall construct uniformly almost flat Bott generators for a simply connected complete Riemannian
manifold of nonpositive sectional curvature, and define a Bott map from the $K$-theory of the $\ell^p$-Roe
algebra to the $K$-theory of thetwisted $\ell^p$-Roe algebra and another Bott map between the
$K$-theory of corresponding $\ell^p$-localization algebras.
We show that the Bott map from the $K$-theory of the $\lp$-localization algebra to the $K$-theory of the twisted
$\lp$-localization algebra is an isomorphism (Theorem \ref{twisted-bott}).
%

Let $M$ be a simply connected complete Riemannian manifold of nonpositive sectional curvature.
As remarked at the beginning of Section 4, without loss of generality, we assume in the following
$\dim (M)=2n$ for some integer $n>0$. Recall that $\cA=C_0(M,\Cliff_{\bbc}(TM))$ is the $C^*$-algebra of
continuous sections of the complex Clifford algebra bundle
$\Cliff_{\bbc}(TM)$ of the tangent bundle of $M$ vanishing at infinity.
Let $\cB:=C_b(M, \cliff_{\bbc}(TM))$ be the $C^*$-algebra of all continuous bounded sections of $\Cliff_{\bbc}(TM)$.

We can consider $\cA=C_0(M,\Cliff_{\bbc}(TM))$ and $\cB=C_b(M, \cliff_{\bbc}(TM))$ as $L^p$-operator algebras, acting on
$L^p(M, \Cliff_{\bbc}(TM))$, the $L^p$-space of the locally measurable sections of the Hilbert space bundle $\Cliff_{\bbc}(TM)$. Since
$\cA$ and $\cB$ act on $L^p(M, \Cliff_{\bbc}(TM))$ by point-wise multiplication, both algebras have equivalent norms for different
$p\in (1, \infty)$. Hence the $K$-theory of both $\cA$ and $\cB$ does not depend on $p$. In particular, the Bott element is still
a generator for $K_0(\cA)$, if $\cA$ is viewed as an $L^p$-operator algebra.

Let $x\in M$. For any $z\in M$, let $\sigma: [0, 1]\to M$ be the unique geodesic such that
$$\sigma(0)=x,\qquad\sigma(1)=z.$$
Let $v_x(z):=\frac{\sigma'(1)}{\|{\sigma'(1)}\|} \in T_z M$.  For any $c>0$, take a continuous function
$\phi_{x,c}:M\to [0,1]$ satisfying
\begin{equation}
\phi_{x,c}(z)=\left\{
\begin{array}{ccc}
0, & \mbox{if}\ d(x,z)\leq \frac{c}{2};  \\
1, & \mbox{if}\ d(x,z)\geq c.
\end{array}\right.
\end{equation}
For any $z\in M$, let
$$ f_{x,c}(z):=\phi_{x,c}(z)\cdot v_x(z)\in T_zM. $$
Then $f_{x,c}\in \calb$. The following result describes certain ``uniform almost flatness" of
the functions $f_{x, c}$ ($x\in M$, $c>0$).

\begin{lem}\label{flat}For any $R>0$ and $\varepsilon>0$, there exist a constant $c>0$ and a family of continuous
functions $\{\phi_{x,c}\}_{x\in M}$  satisfying the above condition (1) such that, if $d(x,y)<R$, then
$$\sup_{z\in M}\|{f_{x,c}(z)-f_{y,c}(z)}\|_{T_zM}< \varepsilon. $$
\end{lem}
\begin{proof}Let $c=\frac{2R}{\varepsilon}$. For any $x\in M$, define $\phi_{x,c}:M\to [0,1]$ by
\[
\phi_{x,c}(z)=\left\{
\begin{array}{ll}
0, & \mbox{ if }  d(x,z)\leq \frac{R}{\varepsilon};        \\
\frac{\varepsilon}{R}d(x, z)-1,   &  \mbox{ if }  \frac{R}{\varepsilon}\leq d(x,z)\leq \frac{2R}{\varepsilon};  \\
1, & \mbox{ if }  d(x,z)\geq \frac{2R}{\varepsilon}.
\end{array}\right.
\]
Let $x,y\in M$ such that $d(x,y)<R$. Then we have several cases for the position of $z\in M$ with respect to
$x, y$.
\par
Consider the case where $d(x,z)>c=\frac{2R}{\varepsilon}$ and $d(y,z)>c=\frac{2R}{\varepsilon}$.
Since $\phi_{x,c}(z)=\phi_{y,c}(z)=1$, we have
$$f_{x,c}(z)-f_{y,c}(z)=v_x(z)-v_y(z).$$
Without loss of generality, assume $d(x, z)\leq d(y, z)$. Then there exists a unique point $y'$ on the unique
geodesic connecting $y$ and $z$ such that $d(y', z)=d(x, z)$. Then $d(y', y)<R$ since $d(x, y)<R$, so that
$d(x, y')<2R$.
\par
Let $\exp^{-1}_z:M\to T_zM$ denote the inverse of the exponential map
$$\exp_z: T_zM\to M $$
at $z\in M$. Then we have
\\ \indent ($\alpha$) $\|\exp^{-1}_z(x)\|=d(x,z)=d(y',z)=\|\exp^{-1}_z(y')\|>c=\frac{2R}{\varepsilon}$;
\\ \indent ($\beta$)  $\|\exp^{-1}_z(x)-\exp^{-1}_z(y')\|\leq d(x,y')<2R$, since $M$ has nonpositive sectional curvature;
\\ \indent ($\gamma$) $v_x(z)=-\frac{\exp^{-1}_z(x)}{\|{\exp^{-1}_z(x)}\|}$
                    and $v_y(z)=-\frac{\exp^{-1}_z(y')}{\|{\exp^{-1}_z(y')}\|}$.
\\
Hence, for any $z\in M$, we have
$$\|{f_{x,c}(z)-f_{y,c}(z)}\| = \|{v_x(z)-v_y(z)}\| <2R/(2R/\varepsilon)=\varepsilon  $$
whenever $d(x, y)<R$. Similarly, we can check the inequality in other cases where $z\in M$ satisfies either
$d(x,z)\leq c$ or $d(y,z)\leq c$.
\end{proof}

Now let's consider the short exact sequence
$$
0\longrightarrow \cA\longrightarrow\cBp \stackrel{\pi}{\longrightarrow} \cBp/\cA \longrightarrow0,
$$
where $\cA=C_0(M,\cliff_{\bbc}(TM))$ and $\cBp=C_b(M,\cliff_{\bbc}(TM))$.
For any $f_{x, c}$ ($x\in M$, $c>0$) constructed above, it is easy to see that
$[f_{x,c}]:=\pi(f_{x,c})$ is invertible in $\cBp/\cA$ with its inverse $[-f_{x,c}]$. Thus
$[f_{x,c}]$ defines an element in $K_1(\cBp/\cA)$. With the help of the index map
$$
\partial: K_1(\cBp/\cA)\to K_0(\cA),
$$
we obtain an element $\partial([f_{x,c}])$ in
$$
K_0(\cA)=K_0\Big( C_0(M,\cliff_{\bbc}(TM)) \Big) \cong K_0 \Big( C_0(\mathbb{R}^{2n})\otimes \cM_{2^n}(\mathbb{C}) \Big)
        \cong\mathbb{Z}.
        $$
It follows from the construction of $f_{x,c}$ that, for every $x\in M$ and $c>0$, $\partial([f_{x,c}])$ is just the Bott generator of $K_0(\cA)$.

The element $\partial([f_{x,c}])$ can be expressed explicitly as follows. Let
\begin{align*}
W_{x,c}&=\matr{1}{f_{x,c}}{0}{1}\matr{1}{0}{f_{x,c}}{1}\matr{1}{f_{x,c}}{0}{1}\matr{0}{-1}{1}{0} \; ,\\
b_{x,c}&=W_{x,c}\matr{1}{0}{0}{0}W_{x,c}^{-1} \; ,\\
b_0&=\matr{1}{0}{0}{0}\; .
\end{align*}

Then both $b_{x,c}$ and $b_0$ are idempotents in $\cM_2(\cA^+)$, where $\cA^+$ is the algebra jointing a unit
to $\cA$.  It is easy to check that
$$
b_{x,c}-b_0\in C_c(M,\cliff_{\bbc}(TM))\otimes \cM_2(\mathbb{C}),
$$
the algebra of $2\times 2$ matrices  of compactly supported continuous functions, with
$$ \supp(b_{x,c}-b_0)\subset B_M(x,c):=\{z\in M:d(x,z)\leq c\}, $$
where for a matrix $a=\matr{a_{11}}{a_{12}}{a_{21}}{a_{22}}$ of functions on $M$ we define the support of $a$ by
$$\supp (a)= \bigcup^2_{i, j=1} \supp(a_{i, j}).$$
Now we have the explicit expression
$$ \partial([f_{x,c}])=[b_{x,c}]-[b_0]\in K_0(\cA). $$

\begin{lem}[Uniform almost flatness of the Bott generators]\label{bottflat}
The family of idempotents $\{b_{x,c}\}_{x\in M,c>0}$ in $\cM_2(\cA^+)=C_0(M,\cliff_{\bbc}(TM))^+\otimes \cM_2(\mathbb{C})$
constructed above are uniformly almost flat in the following sense:
\\ \indent for any $R>0$ and $\varepsilon>0$, there exist $c>0$ and a family of continuous functions
$\Big\{ \phi_{x,c}: \; M\to [0,1] \Big\}_{x\in M}$ such that, whenever $d(x,y)<R$, we have
$$ \sup_{z\in M} \|{b_{x,c}(z)-b_{y,c}(z)}\|_{\cliff_\bbc(T_z M)\otimes \cM_2(\mathbb{C})}<\varepsilon, $$
where $b_{x, c}$ is defined via $W_{x, c}$ and $f_{x, c}=\phi_{x, c} v_x$ as above, and
$\cliff_\bbc(T_z M)$ is the complexified Clifford algebra of the tangent space $T_z M$.
\end{lem}
\begin{proof} Straightforward from Lemma \ref{flat}.    \end{proof}

It would be convenient to introduce the following notion:
\par
\begin{defn} For $R>0, \varepsilon>0, c>0$, a family of idempotents $\{b_x\}_{x\in M}$ in
$\cM_2(\cA^+)=C_0(M,\cliff_{\bbc}(TM))^+\otimes \cM_2(\mathbb{C})$ is said to be $(R, \varepsilon; c)$-flat if
\begin{enumerate}
\item for any  $x, y\in M$ with $d(x, y)<R$ we have
      $$
      \sup_{z\in M} \|{b_{x}(z)-b_{y}(z)}\|_{\cliff_\bbc(T_z M)\otimes \cM_2(\mathbb{C})}<\varepsilon.
      $$
\item $b_{x}-b_0\in C_c(M,\cliff_{\bbc}(TM))\otimes \cM_2(\mathbb{C})$ and
      $$
      \supp(b_{x}-b_0)\subset B_M(x,c):=\{z\in M:d(x,z)\leq c\}.
      $$
\end{enumerate}
\end{defn}

%
%

{\bf Construction of the Bott map $\beta_*$:}

Now we shall use the above almost flat Bott generators for
$$
K_0(\cA)=K_0\Big(C_0(M,\cliff_{\bbc}(TM))\Big)
$$
 to construct a ``Bott map"
$$
\beta_*: \; K_*(\cp(\pdg))\to K_*(\cspace).
$$

To begin with, we give a representation of $\cp(\pdg)$ on $\lp(\Gamma_d,\lp)$, where $\Gamma_d$ is the
countable dense subset of $\pdg$ as in the definition of $\cspace$.

Let $\calgp(\pdg)$ be the algebra of functions
$$
Q: \; \Gamma_d\times \Gamma_d \to \cK_p
$$
such that
\\ \indent (1) there exists $C>0$ such that $\|Q(x,y)\|\leq C$ for all $x,y\in\Gamma_d$;
\\ \indent (2) there exists $R>0$ such that $Q(x,y)=0$ whenever $d(x,y)>R$;
\\ \indent (3) there exists $L>0$ such that for every $z\in P_d(\Gamma)$, the number of elements
in the following set
$$\{ (x, y)\in\Gamma_d\times \Gamma_d: \; d(x, z)\leq 3R, \; d(y, z)\leq 3R, \; Q(x,y)\neq 0 \}$$
is less than $L$.
\par
The product structure on $\calgp(\pdg)$ is defined by
$$
(Q_1Q_2)(x, y)=\sum_{z\in \Gamma_d} Q_1(x, z) Q_2(z, y).
$$
The algebra $\calgp(\pdg)$ acts on $\lp(\Gamma_d,\lp)$. The operator norm completion
of $\calgp(\pdg)$ with respect to this action is isomorphic to $\cp(\pdg)$ when $\Gamma$ has bounded geometry.

Note that $\cp(\pdg)$ is stable in the sense that $\cp(\pdg)\cong \cp(\pdg)\otimes_p \cM_k(\mathbb{C})$ for all natural
number $k$. Any element in $K_0(\cp(\pdg))$ can be expressed as the difference of the $K_0$-classes of two idempotents
in $\cp(\pdg)$. To define the Bott map
$$\beta_*: \; K_0(\cp(\pdg))\to K_0(\cspace),$$
we need to specify the value
$\beta_*([P])$ in $K_0(\cspace)$ for any idempotent $P\in \cp(\pdg)$.

Now let $P\in \cp(\pdg)\subseteq \cB(\lp(\Gamma_d,\lp))$ be an idempotent. Denote $\norm{P}=N$. For any $0<\varepsilon_1<1/100$,
take an element $Q\in \calgp(\pdg)$ such that
$$
\|P-Q\|< \frac{\varepsilon_1}{2N+2}.
$$
Then $\|Q\|<\|P-Q\|+\|P\|<N+1$, hence
$$
\|Q-Q^2\|\leq\|Q-P\|+\|P\|\|P-Q\|+\|P-Q\|\|Q\|<\varepsilon_1
$$
and there is $R_{\varepsilon_1}>0$ such that $Q(x, y)=0$
whenever $d(x, y)>R_{\varepsilon_1}$.
For any $\varepsilon_2>0$, take by Lemma \ref{bottflat} a family of $(R_{\varepsilon_1}, \varepsilon_2; c)$-flat idempotents
$\{b_x\}_{x\in M}$ in $\cM_2(\cA^+)$ for some $c>0$. Define
$$
\widetilde{Q}, \; \widetilde{Q}_0: \; \Gamma_d\times\Gamma_d \to \cK_p\otimes_p\cA^+  \otimes_p \cM_2(\mathbb{C})
$$
by
$$
\widetilde{Q}(x,y)=Q(x,y)\otimes b_x
$$
and
$$
\widetilde{Q}_0(x,y)=Q(x,y)\otimes b_0,
$$
respectively, for all $(x, y)\in \Gamma_d\times \Gamma_d$, where $b_0=\matr{1}{0}{0}{0}$. Then
$$
\widetilde{Q}, \widetilde{Q}_0 \in \calgp(\pdg, \cA^+\otimes_p \cM_2(\mathbb{C}))
            \cong \calgp(\pdg,\cA^+)\otimes_p \cM_2(\mathbb{C})
$$
and
$$
\widetilde{Q}-\widetilde{Q}_0 \in \calgp(\pdg,\cA)\otimes_p \cM_2(\mathbb{C}).
$$

Since $\Gamma$ has bounded geometry, by the almost flatness of the Bott generators (Lemma \ref{bottflat}), we can choose $\varepsilon_1$ and $\varepsilon_2$
small enough to obtain $\widetilde{Q}, \widetilde{Q}_0$ as constructed above such that $\|\widetilde{Q}^2-\widetilde{Q}\|<1/5$ and $\|\widetilde{Q}_0^2-\widetilde{Q}_0\|<1/5$.

It follows that the spectrum of either $\widetilde{Q}$ or $\widetilde{Q}_0$ is contained in  disjoint
neighborhoods $S_0$ of $0$ and $S_1$ of $1$ in the complex plane. Let $f: S_0\sqcup S_1\to \mathbb{C}$ be the holomorphic
function such that $f(S_0)=\{0\}, f(S_1)=\{1\}$. Let $\Theta=f(\widetilde{Q})$ and
$\Theta_0=f(\widetilde{Q}_0)$. Then $\Theta$ and $\Theta_0$ are idempotents in $\cp(\pdg,\cA^+)\otimes \cM_2(\mathbb{C})$  with
$$
\Theta-\Theta_0\in \cspace\otimes \cM_2(\mathbb{C}).
$$
Note that $\cspace\otimes \cM_2(\mathbb{C})$ is a closed two-sided ideal of $\cp(\pdg,\cA^+)\otimes \cM_2(\mathbb{C})$.

At this point we need to recall the {\em difference construction} in $K$-theory of Banach algebras
introduced by Kasparov-Yu \cite{KY06}. Let $J$ be a closed two-sided ideal of a Banach algebra $B$.
Let $p, q\in B^+$ be idempotents such that $p-q\in J$.
Then a difference element $D(p,q)\in K_0(J)$ associated to the pair $p,q$ is defined as follows. Let
\[
Z(p, q)=\left(
\begin{array}{cccc}
q    &  0   &  1-q  &    0    \\
1-q  &  0   &   0   &    q    \\
0    &  0   &   q   &   1-q   \\
0    &  1   &   0   &    0
\end{array}
\right)
\in \cM_4(B^+).
\]
We have
\[
(Z(p, q))^{-1}=\left(
\begin{array}{cccc}
q    &  1-q &   0   &    0    \\
0    &  0   &   0   &    1    \\
1-q  &  0   &   q   &    0    \\
0    &  q   &  1-q  &    0
\end{array}
\right)
\in \cM_4(B^+).
\]
Define
\[
D_0(p,q)=(Z(p, q))^{-1}\left(
\begin{array}{cccc}
p  &  0  &  0  &  0 \\
0  & 1-q &  0  &  0 \\
0  &  0  &  0  &  0 \\
0  &  0  &  0  &  0
\end{array}
\right)
Z(p, q).
\]
Let
\[
p_1=\left(
\begin{array}{cccc}
1  &  0  &  0  &  0 \\
0  &  0  &  0  &  0 \\
0  &  0  &  0  &  0 \\
0  &  0  &  0  &  0
\end{array}
\right).
\]
Then $D_0(p, q)\in \cM_4(J^+)$ and
$D_0(p, q)=p_1 $ modulo $\cM_4(J)$. We define the difference element
$$
D(p, q):=[D_0(p, q)]-[p_1]
$$
in $K_0(J)$.

Finally, for any idempotent $P\in \cp(\pdg)$  representing an element  $[P]$ in $K_0(\cp(\pdg))$, we define
$$
\beta_*([P])=D(\Theta, \Theta_0)\in K_0(\cspace).
$$
The correspondence $[P]\to \beta_*([P])$ extends to a homomorphism, the Bott map
$$
\beta_*: \; K_0(\cp(\pdg))\to K_0(\cspace).
$$
By using suspension, we similarly define the Bott map
$$\beta_*: \; K_1(\cp(\pdg))\to K_1(\cspace).$$
%

{\bf  Construction of the Bott map $(\beta_L)_*$ :}

Next we shall construct a Bott map for $K$-theory of $\lp$-localization algebras:
$$
(\beta_L)_*: \; K_*(\clp(\pdg))\to K_*(\clspace).
$$

Let $\clalgp(\pdg)$ be the algebra of all bounded, uniformly continuous functions
$$
g: \; \mathbb{R}_+ \to \calgp(\pdg) \subset \cB(\lp(\Gamma_d,\lp))
$$
with the following properties:
\begin{enumerate}
\item there exists a bounded function $R: \mathbb{R}_+\to \mathbb{R}_+$ with
                $\displaystyle\lim_{t\to \infty} R(t)=0$ such that $g(t)(x,y)=0$ whenever $d(x,y)>R(t)$ for every $t$;
\item there exists $L>0$ such that for every $z\in P_d(\Gamma)$, the number of elements
                    in the following set
                    $$\{ (x, y)\in\Gamma_d\times \Gamma_d: \; d(x, z)\leq 3R, \; d(y, z)\leq 3R, \; g(t)(x,y)\neq 0 \}$$
                    is less than $L$ for every $t\in \mathbb{R}_+$.
\end{enumerate}

The $\lp$-localization algebra $\clp(\pdg)$ is isomorphic to the norm completion of \\
$\clalgp(\pdg)$ under the norm
$$
\|g\|_\infty:=\sup_{t\in \mathbb{R}_+} \|g(t)\|
$$
when $\Gamma$ has bounded geometry.  Note that $\clp(\pdg)$ is stable in the sense that
$$\clp(\pdg)\cong \clp(\pdg)\otimes_p \cM_k(\mathbb{C})$$
for all natural number $k$. Hence, any element in
$K_0(\clp(\pdg))$ can be expressed as the difference of the $K_0$-classes of two idempotents in $\clp(\pdg)$.
To define the Bott map
$$(\beta_L)_*: \; K_0(\clp(\pdg))\to K_0(\clspace), $$
we need to specify the value
$(\beta_L)_*([g])$ in $K_0(\clspace)$ for any idempotent $g\in \clp(\pdg)$ representing an element
$[g]\in K_0(\clp(\pdg))$.

Now let $g\in \clp(\pdg)$ be an idempotent with $\norm{g}=N$. For any $0<\varepsilon_1<1/100$,
take an element $h\in \clalgp(\pdg)$ such that
$$
\|g-h\|_\infty< \frac{\varepsilon_1}{2N+2}.
$$
Then $\norm{h-h^2}_\infty<\varepsilon_1$ and there is a bounded function $R_{\varepsilon_1}(t)>0$
with $\displaystyle\lim_{t\to \infty} R_{\varepsilon_1}(t)=0$ such that
$h(t)(x, y)=0$ whenever $d(x, y)>R_{\varepsilon_1}(t)$ for every $t$. Let $\widetilde{R}_{\varepsilon_1}=\sup_{t\in \mathbb{R}_+} R(t)$.
For any $\varepsilon_2>0$, take by Lemma \ref{bottflat} a family of $(\widetilde{R}_{\varepsilon_1}, \varepsilon_2; c)$-flat idempotents
$\{b_x\}_{x\in M}$ in $\cM_2(\cA^+)$ for some $c>0$. Define
$$
\widetilde{h},\; \widetilde{h}_0: \; \mathbb{R}_+\to \calgp(\pdg,\cA^+)\otimes_p \cM_2(\mathbb{C})
$$
by
$$
\Big( \widetilde{h}(t) \Big) (x, y) = \Big( h(t)(x,y) \Big) \otimes_p b_x
                        \in \cK_p \otimes_p \cA^+\otimes_p \cM_2(\mathbb{C}),
$$
$$
\Big( \widetilde{h}_0(t) \Big) (x, y) = \Big( h(t)(x,y) \Big) \otimes_p \scriptsize{\matr{1}{0}{0}{0}}
                        \in \cK_p \otimes_p \cA^+\otimes \cM_2(\mathbb{C})
$$
for each $t\in \mathbb{R}_+$. Then we have
$$
\widetilde{h}, \; \widetilde{h}_0 \; \in \;  \clalgp(\pdg,\cA^+)\otimes_p \cM_2(\mathbb{C})
$$
and
$$
\widetilde{h}-    \widetilde{h}_0 \; \in  \; \clalgp(\pdg,\cA)\otimes_p \cM_2(\mathbb{C}).
$$
Since $\Gamma$ has bounded geometry, by the almost flatness of the Bott generators,
we can choose $\varepsilon_1$ and $\varepsilon_2$
small enough to obtain $\widetilde{h}, \widetilde{h}_0$, as constructed above, such that \\
$\|\widetilde{h}^2-\widetilde{h}\|_\infty<1/5$ and $\|\widetilde{h}_0^2-\widetilde{h}_0\|<1/5$.
The spectrum of either $\widetilde{h}$ or $\widetilde{h}_0$ is contained in disjoint neighborhoods
$S_0$ of $0$ and $S_1$ of $1$ in the complex plane.
Let $f: S_0\sqcup S_1\to \mathbb{C}$ be the
function such that $f(S_0)=\{0\}, f(S_1)=\{1\}$. Let $\eta=f(\widetilde{h})$ and
$\eta_0=f(\widetilde{h}_0)$. Then $\eta$ and $\eta_0$ are idempotents in
$\clp(\pdg,\cA^+)\otimes_p \cM_2(\mathbb{C})$  with
$$
\eta-\eta_0\in \clspace\otimes_p \cM_2(\mathbb{C}).
$$
Thanks to the difference construction, we define
$$
(\beta_L)_*([g])=D(\eta, \eta_0)\in K_0(\clspace).
$$
This correspondence $[g]\mapsto (\beta_L)_*([g])$ extends to a homomorphism, the Bott map
$$
(\beta_L)_*: \; K_0(\clp(\pdg))\to K_0(\clspace).
$$
By suspension, we similarly define
$$
(\beta_L)_*: \; K_1(\clp(\pdg))\to K_1(\clspace).
$$
This completes the construction of the Bott map $(\beta_L)_*$.

It follows from the constructions of $\beta_*$ and $(\beta_L)_*$, we
have the following commuting diagram
\[
\xymatrix{
  K_*(\clp(\pdg)) \ar[d]_{e_*} \ar[r]^{(\beta_L)_*}
                &  K_*(\clspace)   \ar[d]^{e_*} \\
  K_*(\cp(\pdg))   \ar[r]^{\beta_*} & K_*(\cspace)
     }
\]

\begin{thm}\label{twisted-bott} For any $d\geq 0$, the Bott map
$$
(\beta_L)_*: \; K_*(\clp(\pdg))\to K_*(\clspace)
$$
is an isomorphism.
\end{thm}
\begin{proof} Note that $\Gamma$ has bounded geometry, and both the $\lp$-localization algebra and
the twisted $\lp$-localization algebra have strong Lipschitz homotopy invariance at the $K$-theory level.
By a Mayer-Vietoris sequence argument and induction on the dimension of the skeletons \cite{Yu97, CW02}, the general
case can be reduced to the $0$-dimensional case, namely, if $D\subset \pdg$ is a
$\delta$-separated subspace (meaning $d(x,y)\geq \delta$ if $x\neq y\in D$) for some $\delta>0$, then
$$(\beta_L)_*: \; K_*(\clp(D))\to K_*(\clp(D,\cA))$$
is an isomorphism. But this follows from the facts that
$$K_*(\clp(D))\cong \prod_{\gamma\in D} K_*(\clp(\{\gamma\})),$$
$$K_*(\clp(D,\cA))\cong\prod_{\gamma\in D} K_*(\clp(\{\gamma\}, \cA))$$
and that $(\beta_L)_*$ restricts to an isomorphism from $K_*(\clp(\{\gamma\}))\cong K_*(\cK_p)$ to
$$K_*(\clp(\{\gamma\}, \cA)) \cong K_*(\cK_p\otimes \cA)$$
at each $\gamma\in D$ by the classical Bott periodicity.
\end{proof}

\section{Proof of the Main Theorem}
\begin{proof}[Proof of Theorem \ref{main}] We have the commuting diagram
\[
\xymatrixcolsep{3pc}\xymatrix{
            \displaystyle\lim_{d\to\infty}K_*(\clp(P_d(\Gamma))) \ar[d]_{e_*}\ar[r]^{(\beta_L)_*}_{\cong}
            &   \displaystyle\lim_{d\to\infty}K_*(\clspace)  \ar[d]^{e_*}_{\cong}                     \\
\displaystyle\lim_{d\to\infty}K_*(\cp(P_d(\Gamma)))\ar[r]^{\beta_*}
   &       \displaystyle\lim_{d\to\infty}K_*(\cspace).
  }
\]
Hence, $\beta_*\circ e_* =e_*\circ(\beta_L)_*$ It follows from Theorem \ref{twisted-iso}
and Theorem \ref{twisted-bott} that $\beta_*\circ e_*$ is an isomorphism. Consequently, the assembly map
$$
\mu=e_*: \; \displaystyle\lim_{d\to\infty}K_*(\clp(P_d(\Gamma))) \to
            \displaystyle\lim_{d\to\infty}K_*(\cp(P_d(\Gamma)))\cong K_*(\cp(\Gamma))$$
is injective.
\end{proof}

\section*{Acknowledgement}
\mbox{} \quad
The authors are very grateful to the referee for a list of constructive suggestions for corrections, changes or improvements to the original version of this paper.


\begin{thebibliography}{99}
%
%
\bibitem{A}
        Atiyah, M. F., {\em Bott periodicity and the index of elliptic operators},
        Q. J. Math., {\bf 19}(1968), 113-140
\bibitem{ABS}
        Atiyah, M. F., Bott, R., and Shapiro, A., {\em Clifford modules},
        Topology, {\bf 3}, Suppl. 1, (1964), 3--38.
\bibitem{B}
        Blackadar, B., {\em K-Theory for Operator Algebras} (2nd edition),
        Cambridge Univ. Press, 1998.
\bibitem{CE}
        Cheeger, J. and Ebin, D. G., {\em Comparison theorems in Riemannian geometry},
        North-Holland Publishing Company, Amsterdam, 1975.
\bibitem{CW02}
         Chen, X. and Wang, Q., {\em Localization algebras and duality}, J. London Math.
         Soc., {\bf 66}(2)(2002) 227--239.
\bibitem{CWY} Chen, X., Wang, Q., and Yu, G., {\em The coarse Novikov conjecture and Banach spaces with property (H)},
            Journal of Functional Analysis, {\em 268} (2015) 2754-2786.

\bibitem{Chung}
        Chung, Y. C., {\em Dynamical complexity and $K$-theory of $L^p$ operator crossed products}, Journal of Toppology and Analysis, 2020, online ready. Arxiv:1611.09000.

\bibitem{ChungLi18}
        Chung, Y. C. and Li. K., {\em Rigidity of $\lp$ Roe-type algebras}, Bull. Lond. Math. Soc. {\bf 50} (2018), no. 6, 1056-1070.
\bibitem{ChungLi19}
        Chung, Y. C. and Li, K., {\em Structure and $K$-theory of $\ell^p$ uniform Roe algebras}, Journal of Noncommutative Geometry, to appear,
        Arxiv: 1904.07050.

\bibitem{ChungNowak}
        Chung, Y. C. and Nowak, P., {\em Expanders are counterexamples to the coarse $p$-Baum-Connes conjecture}, preprint (2018). arXiv:1811.10457.

%

\bibitem{DefantFloret}
        Defant, A. and Floret, K., {\em Tensor norms and operator ideals}, North-Holland Mathematics Studies, North-Holland Pulishing Co., 1993.
\bibitem{Dr03}
        Dranishnikov, A. N., {\em On hypersphericity of manifolds with finite asymptotic dimension},
            Trans. Amer. Math. Soc. {\bf 355}(1)(2003), 155--167.
\bibitem{Fell-Doran}
    Fell, J. M. G.; Doran, R. S. {\em Representations of $*$-algebras, locally compact groups, and Banach $*$-algebraic bundles. Vol. 1. Basic representation theory of groups and algebras}. Pure and Applied Mathematics, 125. Academic Press, Inc., Boston, MA, 1988. xviii+746 pp. ISBN: 0-12-252721-6.
\bibitem{Figiel-Iwaniec-Pelczynski}
        Figiel, T.; Iwaniec, T.; Pelczynski, A. {\em Computing norms and critical exponents of some operators in $L^p$-spaces}. Studia Math. 79 (1984), no. 3, 227-274.
\bibitem{Grom}
        Gromov, M., {\em Asymptotic invariants for infinite groups}, Vol 2,
        Proc. 1991 Sussex Conference on Geometry Group Theory, LMS Lecture
        Note Ser. 182, Academic Press, New York, 1993.
\bibitem{HR}
        Higson, N. and Roe, J., {\em On the coarse Baum-Connes conjecture},
        in: S. Ferry, A. Ranicki and J. Rosenberg (eds), Proc. 1993 Oberwolfach
        Conference in the Novikov Conjecture, London Math. Soc. Lecture Note
        Series 227, Cambridge University Press, 1995, PP. 227-254.
\bibitem{HRY}
        Higson, N., Roe, J. and Yu, G., {\em A coarse Mayer-Vietoris principle},
        Math. Proc. Camb. Phil. Soc., {\bf 114}(1993) 85--97.
\bibitem{Kas75}
        Kasparov, G. G., {\em Topological invariants of elliptic operators I: K-homology}.
            Math. USSR Izvestija, {\bf 9} (1975) 751--792.
\bibitem{Kas88}
       Kasparov, G., {\em Equivariant KK-theory and the Novikov conjecture}, Invent. Math., {\bf 91} (1988)147--201.

\bibitem{Kas2013}
       Kasparov, G., {\em On the $L^p$ Novikov and Baum-Connes conjectures}. Lecture slides available at \\ https://www.ms.u-tokyo.ac.jp/~kida/workshops/ask2013/Kasparov.pdf

\bibitem{KY06}
      Kasparov, G. and Yu, G., {\em The coarse geometric Novikov conjecture and uniform convexity}, Advances in Mathematics,
            {\bf 206} (1) (2006) 1-56.

\bibitem{Lang}
        Serge Lang, {\em Fundamentals of Differential Geometry}, GTM 191, 1999, Springer-Verlag New York, Inc.

\bibitem{Lafforgue}
        Lafforgue, V.: {\em Banach $KK$-theory and the Baum-Connes conjecture}, in Proc. of the Int. Congress of Mathematicians,
        Vol. II (Beijing, 2002) (Higher Ed. Press, 2002), 795-812.

\bibitem{Phillips12}
    Phillips, N. C., {\em  Analogs of Cuntz algebras on $L^p$ spaces}, arXiv:1201.4196v1 [math.FA] 20 Jan 2012.
\bibitem{NCP}
    Phillips, N. C., {\em Crossed products of $L^p$ operator algebras and the $K$-theory of Cuntz algebras on $L^p$ spaces},
    arXiv:1309.6406v1 [math.FA] 25 Sep 2013.

\bibitem{QR}
    Qiao, Y. and Roe, J., {\em On the localization algebra of Guoliang Yu}, Forum Math., {\bf 22(4)}:657-665, 2010.
\bibitem{Roe93}
    Roe, J., {\em Coarse cohomology and index theory on complete Riemannian manifolds}, Mem. Amer. Math. Soc.
        {\bf 104} (1993), no. 497, x+90 pp.
\bibitem{Roe96}
    Roe, J., {\em Index theory, coarse geometry, and the topology of manifolds}, CBMS Conference Proceedings 90,
    American Mathematical Society, Providence, R.I., 1996.
\bibitem{Shan}
    Shan, L., {\em An equivariant higher index theory and nonpositively curved manifolds}, Journal of Functional Analysis, Vol. 255, Issue 6, 1480--1496.
\bibitem{SW}
    Shan, L. and Wang, Q., {\em The coarse geometric Novikov conjecture for subspaces of non-positively curved manifolds}, Journal of Functional Analysis, Vol. 248, Issue 2, 448--471.
    
    
\bibitem{Willett-Yu-II}
        Willett, R. and Yu, G., {\em Higher index theory for certain expanders and Gromov monster groups, II}, Adv. Math. 229 (3) (2012) 1762-1803.

\bibitem{WY}
    Willett, R. and Yu, G., {\em Higher Index Theory}, Cambridge studies in advanced mathematics 189, Cambridge University Press, 2020.

\bibitem{Williams}
    Williams, Dana P., {\em Crossed products of $C^*$-algebras}. Mathematical Surveys and Monographs, 134. American Mathematical Society, Providence, RI, 2007. xvi+528 pp.

\bibitem{Yu95}
        Yu, G., {\em Coarse Baum-Connes conjecture}, K-Theory, {\bf 9}(3)(1995)
        199--221.
\bibitem{Yu97}
        Yu, G., {\em Localization algebras and the coarse Baum-Connes conjecture}, K-theory {\bf 11}(1997) 307--318.
\bibitem{Yu98}
        Yu, G. {\em The Novikov conjecture for groups with finite asymptotic dimension}, Annals of Mathematics,
            {\bf 147}(2) (1998), 325-355.

\bibitem{Yu00}
        Yu, G., {\em The coarse Baum-Connes conjecture for spaces which admit
        a uniform embedding into Hilbert space}, Invent. Math., {\bf 139}(2000) 201--240.

\bibitem{Yu05}
        Yu, G., {\em  Hyperbolic groups admit proper affine isometric actions on $l^p$-spaces}. Geom. Funct. Anal., 15(5):1144-1151, 2005.

\bibitem{Yu06}
        Yu, G., {\em Higher index theory of elliptic operators and geometry of groups},
        Proceedings of International Congress of Mathematicians, Madrid, 2006, vol. II, 1623--1639.

\bibitem{ZZ}
        Zhang, J. and Zhou, D., {\em $L^p$ coarse Baum-Connes conjecture and $K$-theory for $L^p$ Roe algebras},
        arXiv:1909.08712.
\end{thebibliography}
\end{document}